\documentclass[11pt,reqno]{amsart}

\usepackage{amssymb,latexsym,amsmath,amsthm,bbm,caption}
\usepackage{graphicx,multicol,enumitem,pgfplots,pdfpages}
\pgfplotsset{compat=1.17}
\usetikzlibrary{arrows.meta,calc,decorations.markings,math,arrows.meta,patterns}
\usepackage{fullpage}

\usepackage{tabstackengine}
\stackMath

\counterwithin{figure}{section}
\counterwithin{table}{section}
\numberwithin{equation}{section}
\newcounter{Theorem}
\numberwithin{Theorem}{section}
\newtheorem{theorem}[Theorem]{Theorem}
\newtheorem{lemma}[Theorem]{Lemma}
\newtheorem{proposition}[Theorem]{Proposition}

\theoremstyle{definition}
\newtheorem{definition}[Theorem]{Definition}
\newtheorem{remark}[Theorem]{Remark}
\newtheorem{example}[Theorem]{Example}

\DeclareMathOperator*{\half}{\textstyle{\frac{1}{2}}}
\DeclareMathOperator*{\quarter}{\textstyle{\frac{1}{4}}}

\newcommand{\mm}{\phantom{-}}

\newcommand{\exactframe}[2]{%
  \fbox{%
    \rule{\dimexpr#1-2\fboxsep-2\fboxrule}{0pt}% width
    \rule{0pt}{\dimexpr#2-2\fboxsep-2\fboxrule}% height
  }%
}

\allowdisplaybreaks

\begin{document}

\title{Stability of the \`A Trous Algorithm Under Iteration}
\author{Brody Johnson}
\address{Department of Mathematics and Statistics, Saint Louis University, St. Louis, MO 63103}
\email{brody.johnson@slu.edu}
\author{Simon McCreary-Ellis}
\address{Department of Mathematics and Statistics, Saint Louis University, St. Louis, MO 63103}
\email{simon.mccrearyellis@slu.edu}

\date{\today}

\maketitle

\begin{abstract}
This paper examines the stability of the \`a trous algorithm under arbitrary iteration in the context of a more general study of shift-invariant filter banks.  The main results describe sufficient conditions on the associated filters under which an infinitely iterated shift-invariant filter bank is stable.  Moreover, it is shown that the stability of an infinitely iterated shift-invariant filter bank guarantees that of any associated finitely iterated shift-invariant filter bank, with uniform bounds.  The reverse implication is shown to hold under an additional assumption on the low-pass filter.  Finally, it is also shown that the separable product of stable one-dimensional shift-invariant filter banks produces a stable two-dimensional shift-invariant filter bank.
\end{abstract}

\section{Introduction} \label{intro}

A dyadic wavelet transform relies on the refinement equations satisfied by the the scaling function, $\varphi$, and wavelet, $\psi$.  Assume that there exist sequences $\lbrace h(n)\rbrace_{n\in \mathbb{Z}}$ and $\lbrace g(n) \rbrace_{n\in \mathbb{Z}}$ so that
\begin{align} 
\frac{1}{2} \varphi (\half x ) &= \sum_{n\in \mathbb{Z}} h(n) \varphi(x-n) \label{refine-phi} \\
\frac{1}{2} \psi (\half x ) &= \sum_{n\in \mathbb{Z}} g(n) \varphi(x-n) \label{refine-psi},
\end{align}

\noindent
with $\sum_{n\in \mathbb{Z}} h(n)=1$ and $\sum_{n\in \mathbb{Z}} g(n)=0$.  Let $(Df)(x)=\sqrt{2}f(2x)$ and $(Tf)(x)=f(x-1)$.  It is not difficult to show that \eqref{refine-phi} and \eqref{refine-psi} imply 
\begin{align}
\langle f, D^{j-1} T^{k} \varphi \rangle &= \sqrt{2} \sum_{n\in \mathbb{Z}} h(n) \langle f, D^{j} T^{2k+n} \varphi \rangle \label{dwt-phi} \\
\langle f, D^{j-1} T^{k} \psi \rangle &= \sqrt{2} \sum_{n\in \mathbb{Z}} g(n) \langle f, D^{j} T^{2k+n} \varphi \rangle, \label{dwt-psi} 
\end{align}

\noindent
where $j,k\in \mathbb{Z}$.  Given the sequence of coefficients $\lbrace \langle f, T^{k}\varphi \rangle \rbrace_{k\in \mathbb{Z}}$, one can thus compute the wavelet coefficients $\lbrace \langle f, D^{j}T^{k} \psi \rangle \rbrace_{k\in \mathbb{Z}}$ for all scales $j<0$ using an infinitely iterated, critically sampled dyadic filter bank.  This recursive scheme is the discrete wavelet transform associated with multiresolution analysis wavelets introduced by Mallat \cite{Mallat1989}, which blends the theory of orthonormal wavelets with the theory of perfect reconstruction filter banks \cite{SmithBarnwell}.

Alternatively, with the help of the identity $D^{-1}T=T^{2}D^{-1}$, one can show that \eqref{refine-phi} and \eqref{refine-psi} also lead to
\begin{align}
\langle f, \sqrt{2}^{j-1} T^{k} D^{j-1}\varphi \rangle &= \sum_{n\in \mathbb{Z}} h(n) \langle f, \sqrt{2}^{j} T^{k+2^{-j}n} D^{j}  \varphi \rangle = \sum_{n\in \mathbb{Z}} (U^{j}h)(n) \langle f, \sqrt{2}^{j} T^{k+n} D^{j}  \varphi \rangle \label{atrous-phi} \\
\langle f, \sqrt{2}^{j-1} T^{k} D^{j-1} \psi \rangle &= \sum_{n\in \mathbb{Z}} g(n) \langle f, \sqrt{2}^{j} T^{k+2^{-j}n} D^{j} \varphi \rangle = \sum_{n\in \mathbb{Z}} (U^{j}g)(n) \langle f, \sqrt{2}^{j} T^{k+n} D^{j}  \varphi \rangle, \label{atrous-psi} 
\end{align}

\noindent
where $j, k\in \mathbb{Z}$ with $j\le 0$ and $U$ acts on a sequence over the integers by inserting a zero between each pair of consecutive terms.  Hence, given the sequence of coefficients $\lbrace \langle f, T^{k}\varphi \rangle \rbrace_{k\in \mathbb{Z}}$, one can compute the wavelet coefficients $\lbrace \langle f, \sqrt{2}^{j} T^{k} D^{j} \psi \rangle \rbrace_{k\in \mathbb{Z}}$ for all scales $j<0$ using an infinitely iterated, non-decimated  filter bank whose filters are appropriately dilated at scale $j$ via $U^{j}$.  This recursive scheme is the \`a trous algorithm introduced by Holschneider, Kronland-Martinet, Morlet, and Tchamitchian \cite{HKMT1989}.  It is obvious from \eqref{atrous-psi} that the \`a trous algorithm respects integer translations in the sense that
$$ \langle T^{\ell} f, \sqrt{2}^{j-1} T^{k} D^{j-1} \psi \rangle = \langle f,  \sqrt{2}^{j-1} T^{k-\ell} D^{j-1} \psi \rangle,$$

\noindent
for $j,k,\ell \in \mathbb{Z}$ with $j\le 0$.  Moreover, it should be clear from \eqref{dwt-psi} that the discrete wavelet transform does not possess this property, which explains the fact that the \`a trous algorithm is sometimes referred to as the shift-invariant discrete wavelet transform.  Another key difference between the two algorithms lies in the fact that the \`a trous algorithm implements an overcomplete (or redundant) representation of the original sequence, while the representation provided by the usual discrete wavelet transform is exact \cite{Shensa1992}.

The \`a trous algorithm is well suited for many applications due to both its redundancy and translation invariance.  Applications of the \`a trous algorithm include the analysis of acoustic signals \cite{KMG1987,EGLFS2005} and the characterization of signals using multiscale edges \cite{MallatZhong1992}.  Moreover, the closely related maximum overlap discrete wavelet transform (MODWT) found application to the estimation of variance for multi-scale analysis of time series \cite{MODWT2000,PercivalWalden2000}.  Researchers have continued to find new applications for these overcomplete and translation invariant versions of the discrete wavelet transform in many areas of study, including sleep apnea \cite{Apnea2008}, signal singularities \cite{Rakowski2015}, facial recognition \cite{FacialRecognition2016}, as well as image and data fusion \cite{RemoteSensing2008, ImageFusion2011, ImageFusion2018, DataFusion2019}.  An often exploited link between the discrete wavelet transform and the \`a trous algorithm is that any perfect reconstruction filter bank for the discrete wavelet transform will yield perfect reconstruction for the \`a trous algorithm, although, in general, the converse is not true.  In fact, one often finds that the filters used in applications of the \`a trous algorithm and related shift-invariant wavelet transforms are actually perfect reconstruction filters associated with the discrete wavelet transform \cite{EGLFS2005, Apnea2008, DataFusion2019}.  It is unclear whether this choice is made based on the performance of the filters for a given application or simply stems from a gap in the mathematical literature regarding the underlying theory of the \`a trous algorithm.

In contrast, the discrete wavelet transform and its relationship with multiresolution analysis wavelets has been extensively studied.  Characterizations of trigonometric polynomial low-pass filters associated with multiresolution analysis wavelets were studied early in the development of wavelet theory by Lawton \cite{Lawton1990,Lawton1991} and Cohen \cite{Cohen1990}, culminating with a complete characterization of low-pass filters due to Gundy \cite{Gundy2000}.  More generally, Han provided separate necessary and sufficient conditions on the low- and high-pass filters under which the associated wavelet gives rise to a Riesz basis for $L^{2}(\mathbb{R})$ \cite{Han2005}.  Subsequently, Bayram and Selesnick initiated a study of the frame bounds of iterated filter banks associated with the discrete wavelet transform which relates the frame bounds of an iterated filter bank to those of the underlying wavelet system in $L^{2}(\mathbb{R})$ \cite{BayramSelesnick2009}.

In the Discussion section of their paper, Bayram and Selesnick proposed the study of iterated filter banks solely in terms of the filters themselves, without reference to any associated scaling function or wavelet \cite{BayramSelesnick2009}.  The authors, in a joint work with Bownik, have further examined the stability of iterated filter banks associated with the discrete wavelet transform, deriving easily verified sufficient conditions on the low- and high-pass filters for stability under arbitrary iteration and with uniform frame bounds \cite{BJME2022}.  The goal of the present work is to carry out a similar study for iterated shift-invariant filter banks such as those used in the \`a trous algorithm.  Ultimately, both necessary and sufficient conditions for the stability of the infinitely iterated \`a trous algorithm are of interest; however, this work will focus on straightforward sufficient conditions leading to a broader class of suitable filters for use with the \`a trous algorithm and its variants.  Owing to the fact that the \`a trous algorithm, by design, implements a highly redundant representation of signals, this study will allow for the use of multiple high-pass filters in conjunction with a single low-pass filter.  Moreover, it will be natural to focus the analysis on an infinitely iterated shift-invariant filter bank, as depicted in Figure \ref{SI-IFB-analysis}.  

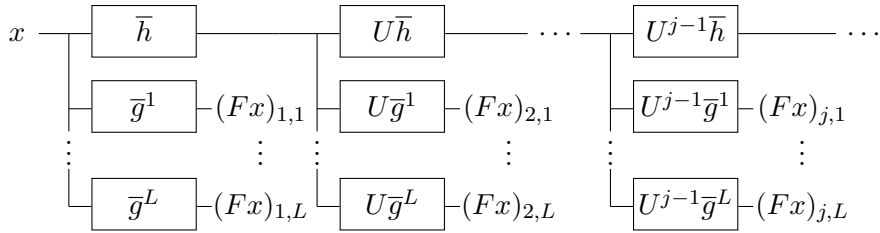
\begin{figure}[hbtp]
\centering
\begin{picture}(120,28)(1,6)
	\put (1,30){$x$}
	\put (5,31){\line(1,0){7}}

	\put (9,31){\line(0,-1){13}}
	\put (8,13.5){\makebox[2mm]{$\vdots$}}
    \put (12,29){\exactframe{14mm}{7mm}}
    \put (12,30){\makebox[14mm]{$\overline{h}$}}
	\put (26,31){\line(1,0){11}}
	\put (37,31){\line(1,0){8}}

	\put (42,31){\line(0,-1){13}}
	\put (41,13.5){\makebox[2mm]{$\vdots$}}
    \put (45,29){\exactframe{14mm}{7mm}}
    \put (45,30){\makebox[14mm]{$U\overline{h}$}}
	\put (59,31){\line(1,0){11}}
	
    \put (71,30){\makebox[6mm]{$\cdots$}}
	\put (77,31){\line(1,0){7}}
	\put (81,31){\line(0,-1){13}}
	\put (80,13.5){\makebox[2mm]{$\vdots$}}
    \put (84,29){\exactframe{14mm}{7mm}}
    \put (84,30){\makebox[14mm]{$U^{j-1}\overline{h}$}}
	\put (98,31){\line(1,0){13}}

    \put (9,13){\line(0,-1){5}}
    \put (9,21){\line(1,0){3}}
    \put (12,19){\exactframe{14mm}{7mm}}
    \put (12,20){\makebox[14mm]{$\overline{g}^{1}$}}
	\put (26,21){\line(1,0){2}}
	\put (30.5,20){\makebox[8mm]{$(Fx)_{1,1}$}}
	\put (33.5,13.5){\makebox[2mm]{$\vdots$}}

	\put (42,13){\line(0,-1){5}}
	\put (42,21){\line(1,0){3}}
    \put (45,19){\exactframe{14mm}{7mm}}
    \put (45,20){\makebox[14mm]{$ U\overline{g}^{1}$}}
	\put (59,21){\line(1,0){2}}
	\put (63.5,20){\makebox[8mm]{$(Fx)_{2,1}$}}
	\put (66.5,13.5){\makebox[2mm]{$\vdots$}}

	\put (81,13){\line(0,-1){5}}
	\put (81,21){\line(1,0){3}}
    \put (84,19){\exactframe{14mm}{7mm}}
    \put (84,20){\makebox[14mm]{$U^{j-1}\overline{g}^{1}$}}
	\put (98,21){\line(1,0){2}}
	\put (102.5,20){\makebox[8mm]{$(Fx)_{j,1}$}}
	\put (105.5,13.5){\makebox[2mm]{$\vdots$}}

	\put (9,8){\line(1,0){3}}
    \put (12,6){\exactframe{14mm}{7mm}}
    \put (12,7){\makebox[14mm]{$\overline{g}^{L}$}}
	\put (26,8){\line(1,0){2}}
	\put (30.5,7){\makebox[8mm]{$(Fx)_{1,L}$}}
	
	\put (42,8){\line(1,0){3}}
    \put (45,6){\exactframe{14mm}{7mm}}
    \put (45,7){\makebox[14mm]{$ U\overline{g}^{L}$}}
	\put (59,8){\line(1,0){2}}
	\put (63.5,7){\makebox[8mm]{$(Fx)_{2,L}$}}

	\put (81,8){\line(1,0){3}}
    \put (84,6){\exactframe{14mm}{7mm}}
    \put (84,7){\makebox[14mm]{$U^{j-1}\overline{g}^{L}$}}
	\put (98,8){\line(1,0){2}}
	\put (102.5,7){\makebox[8mm]{$(Fx)_{j,L}$}}

    \put (111,30){\makebox[8mm]{$\cdots$}}
\end{picture}

\caption{Analysis schematic for an infinitely iterated shift-invariant filter bank.} \label{SI-IFB-analysis}
\end{figure}

\section{Preliminaries} \label{prelim}

\subsection{Mathematical Background and Notation}

Recall that a \emph{frame} for a separable Hilbert space $\mathbb{H}$ is a collection $\lbrace e_{j} \rbrace_{j\in J} \subset \mathbb{H}$, where $J$ is a countable index set, for which there exist constants $0<A\le B<\infty$ (called \emph{frame bounds}) such that for all $x\in \mathbb{H}$,
$$ A \Vert x\Vert_{\mathbb{H}}^{2} \le \sum_{j\in J} \vert \langle x, e_{j}\rangle \vert^{2} \le B \Vert x\Vert_{\mathbb{H}}^{2}.$$

\noindent
A frame is said to be \emph{tight} when it is possible to choose $A=B$ and a tight frame for which $A=B=1$ is called a \emph{Parseval} frame.  If only the right-hand inequality holds, the collection $\lbrace e_{j} \rbrace_{j\in J}$ is called a \emph{Bessel system} and $B$ is called the Bessel bound.  

In this work, the Hilbert space of primary interest is $\ell^{2}(\mathbb{Z})$.  The \emph{Fourier transform} of $x\in \ell^{2}(\mathbb{Z})$ will be denoted by $\hat{x}$ and is defined by
\begin{equation*}
\hat{x}(\xi) = \sum_{n\in\mathbb{Z}} x(n) e^{-2\pi i n \xi}, \quad \text{a.e.} \; \xi \in \mathbb{R}.
\end{equation*}

\noindent
In many cases identities involving such Fourier transforms will be considered on $\mathbb{T}$, often identified with the interval $\lbrack -\half, \half )$.  The \emph{convolution} of sequences $x,y\in \ell^{2}(\mathbb{Z})$ is defined by
$$ (x*y)(k) = \sum_{n\in \mathbb{Z}} y(n) x(k-n), \quad k\in \mathbb{Z},$$

\noindent
which, under the Fourier transform, corresponds to $\widehat{x*y}(\xi) = \hat{x}(\xi) \hat{y}(\xi)$.  Three other operators on $\ell^{2}(\mathbb{Z})$ play a key role in this work.  The \emph{translation} operator $T$ acts on $x\in \ell^{2}(\mathbb{Z})$ by $Tx(k)=x(k-1)$, so that
$$\widehat{Tx}(\xi) = \hat{x}(\xi) e^{-2\pi i \xi}.$$

\noindent
The \emph{involution} of $x\in \ell^{2}$ is denoted by $\overline{x}$ and is defined by $\overline{x}(k) = \overline{x(-k)}$, so that $\hat{\overline{x}}(\xi) = \overline{\hat{x}(\xi)}$.  Finally, the \emph{upsampling} operator $U$ acts on $x\in \ell^{2}(\mathbb{Z})$ by
$$Ux(k) =  \begin{cases} x(m), & k=2m, \\ 0, & \text{otherwise}, \end{cases}$$

\noindent
which leads to $\widehat{Ux}(\xi) = \hat{x}(2\xi)$.

\subsection{Filter Bank Terminology}

The term \emph{filter} will refer to a sequence in $\ell^{2}(\mathbb{Z})$ that acts on a \emph{signal} in $\ell^{2}(\mathbb{Z})$ by convolution.  A generic signal will frequently be denoted by $x$.  The letter $h$ will be used exclusively to represent \emph{low-pass} filters, which are assumed to satisfy $\hat{h}(0)=1$ and $\hat{h}(\half)=0$.  Similarly, the letter $g$ will be reserved for \emph{high-pass} filters, which are assumed to satisfy $\hat{g}(0)=0$.  When multiple high-pass filters are employed, they will be denoted by $g^{1}, g^{2}, \ldots, g^{L}$, with $L\in \mathbb{N}$.  For $j\in \mathbb{N}$, the \emph{iterated low-pass filter of order $j$} associated with a low-pass filter $h$ is defined as
\begin{equation} \label{alpha-j}
h_{j} = h * Uh * \cdots * U^{j-1}h.
\end{equation}

\noindent
The \emph{iterated high-pass filters of order $j$} associated with low-pass filter $h$ and high-pass filters $g^{1}, g^{2}, \ldots, g^{L}$ are defined by
\begin{equation} \label{beta-j}
g^{\ell}_{j} = h * U h * \cdots * U^{j-2}h * U^{j-1} g^{\ell} = h_{j-1} * U^{j-1} g^{\ell}.
\end{equation}

\noindent
The \emph{filter bank analysis operator} associated with the filters $h, g^{1}, g^{2}, \ldots, g^{L} \in \ell^{2}(\mathbb{Z})$ and acting on $x\in \ell^{2}(\mathbb{Z})$ is the mapping
$$F:\ell^{2}(\mathbb{Z}) \rightarrow \bigoplus_{\underset{1\le \ell \le L}{j\in \mathbb{N}}} \ell^{2}(\mathbb{Z})$$

\noindent
given by
\begin{equation*}
F: x \mapsto Fx := \lbrace (Fx)_{j,\ell} \rbrace_{j\in \mathbb{N}, \, 1\le \ell \le L}.
\end{equation*}

\noindent
Referring to Figure \ref{SI-IFB-analysis}, observe that $(Fx)_{j,\ell}$ is given by
$$ (Fx)_{j,\ell}(k) = (x*\overline{g}^{\ell}_{j})(k) = \langle x, T^{k} g^{\ell}_{j} \rangle,$$

\noindent
so that the components of the filter bank analysis operator can be interpreted as the frame coefficients of $x$ relative to the integer translates of the iterated high pass filters.

\begin{definition} \label{stable-def}
The infinitely iterated shift-invariant filter bank defined by $h, g^{1}, g^{2}, \ldots, g^{L}$ is said to be \emph{stable} when the collection
$$ \lbrace T^{k} g^{\ell}_{j} : j\in \mathbb{N}, \; k\in \mathbb{Z}, \; 1\le \ell \le L \rbrace$$

\noindent
constitutes a frame for $\ell^{2}(\mathbb{Z})$. 
\end{definition}

The following equivalent formulation of stability will be used frequently in this work.  The infinitely iterated shift-invariant filter bank generated by $h, g^{1}, g^{2}, \ldots, g^{L}$ is stable if there exist constants $0<A\le B<\infty$ such that for all $x\in \ell^{2}(\mathbb{Z})$,
\begin{equation} \label{frame-def}
A \Vert x\Vert^{2} \le \sum_{\ell =1}^{L} \sum_{j=1}^{\infty} \Vert (Fx)_{j,\ell}\Vert^{2} \le B \Vert x \Vert^{2}.
\end{equation}

\noindent
An infinitely iterated filter bank for which only the upper bound of \eqref{frame-def} holds will be referred to as \emph{Bessel}.  Similarly, if it is possible to choose $A=B$ in \eqref{frame-def}, the infinitely iterated filter bank will be referred to as \emph{tight} and, if $A=B=1$, the infinitely iterated filter bank is said to be \emph{Parseval}.  When an infinitely iterated filter bank is Bessel, the filter bank synthesis operator
$$ F^{*}: \bigoplus_{\underset{1\le \ell \le L}{j\in \mathbb{N}}} \ell^{2}(\mathbb{Z}) \rightarrow \ell^{2}(\mathbb{Z})$$

\noindent
given by
$$ F^{*} \left (\lbrace c_{j,\ell} \rbrace_{j\in \mathbb{N}, 1\le \ell \le L} \right ) = \sum_{\ell=1}^{L} \sum_{j=1}^{\infty} c_{j,\ell}*g^{\ell}_{j}$$

\noindent
is well-defined and bounded.  As indicated by the notation, $F^{*}$, the synthesis operator is adjoint to the analysis operator $F$.  The schematic for the filter bank synthesis operator is depicted in Figure \ref{SI-IFB-synthesis}.

\begin{figure}[hbtp]
\centering
\begin{picture}(134,30)(2,8)
    \put (2.25,30){\makebox[6mm]{$\cdots$}}
	\put (8,31){\line(1,0){9}}
    \put (17,29){\exactframe{14mm}{7mm}}
    \put (17,30){\makebox[14mm]{$U^{j-1}h$}}
	\put (31,31){\line(1,0){4}}
    \put (37,31){\circle{4}}
	\put (35,30){\makebox[4mm]{$+$}}
	\put (39,31){\line(1,0){4}}
    \put (43.25,30){\makebox[6mm]{$\cdots$}}
	\put (49,31){\line(1,0){9}}
    \put (58,29){\exactframe{14mm}{7mm}}
    \put (58,30){\makebox[14mm]{$Uh$}}
	\put (72,31){\line(1,0){4}}
    \put (78,31){\circle{4}}
	\put (76,30){\makebox[4mm]{$+$}}
	\put (80,31){\line(1,0){13}}
    \put (93,29){\exactframe{14mm}{7mm}}
    \put (93,30){\makebox[14mm]{$h$}}
	\put (107,31){\line(1,0){4}}
    \put (113,31){\circle{4}}
	\put (111,30){\makebox[4mm]{$+$}}
	\put (115,31){\line(1,0){3}}
	\put (120,30){\makebox[14mm]{$F^{*}(\lbrace c_{j,\ell}\rbrace)$}}

	\put (7,20){\makebox[8mm]{$c_{j,1}$}}
	\put (15,21){\line(1,0){2}}
    \put (17,19){\exactframe{14mm}{7mm}}
    \put (17,20){\makebox[14mm]{$U^{j-1}g^{1}$}}
	\put (31,21){\line(1,0){4}}
	\put (37,23){\line(0,1){6}}
	\put (37,21){\circle{4}}
	\put (37,18){\line(0,1){1}}
	\put (35,20){\makebox[4mm]{$+$}}
	\put (48,20){\makebox[8mm]{$c_{2,1}$}}
	\put (56,21){\line(1,0){2}}
    \put (58,19){\exactframe{14mm}{7mm}}
    \put (58,20){\makebox[14mm]{$ Ug^{1}$}}
	\put (72,21){\line(1,0){4}}
	\put (78,23){\line(0,1){6}}
	\put (78,21){\circle{4}}
	\put (78,18){\line(0,1){1}}
	\put (76,20){\makebox[4mm]{$+$}}
	\put (83,20){\makebox[8mm]{$c_{1,1}$}}
	\put (91,21){\line(1,0){2}}
    \put (93,19){\exactframe{14mm}{7mm}}
    \put (93,20){\makebox[14mm]{$g^{1}$}}
	\put (107,21){\line(1,0){4}}
	\put (113,23){\line(0,1){6}}
	\put (113,21){\circle{4}}
	\put (113,18){\line(0,1){1}}
	\put (111,20){\makebox[4mm]{$+$}}
    
    \put (10,13){\makebox[2mm]{$\vdots$}}
    \put (36,13){\makebox[2mm]{$\vdots$}}
    \put (51,13){\makebox[2mm]{$\vdots$}}
    \put (77,13){\makebox[2mm]{$\vdots$}}
    \put (86,13){\makebox[2mm]{$\vdots$}}
    \put (112,13){\makebox[2mm]{$\vdots$}}

	\put (7,9){\makebox[8mm]{$c_{j,L}$}}
	\put (15,10){\line(1,0){2}}
    \put (17,8){\exactframe{14mm}{7mm}}
    \put (17,9){\makebox[14mm]{$U^{j-1}g^{L}$}}
	\put (31,10){\line(1,0){6}}
	\put (37,10){\line(0,1){2}}
	\put (48,9){\makebox[8mm]{$c_{2,L}$}}
	\put (56,10){\line(1,0){2}}
    \put (58,8){\exactframe{14mm}{7mm}}
    \put (58,9){\makebox[14mm]{$ Ug^{L}$}}
	\put (72,10){\line(1,0){6}}
	\put (78,10){\line(0,1){2}}
	\put (83,9){\makebox[8mm]{$c_{1,L}$}}
	\put (91,10){\line(1,0){2}}
    \put (93,8){\exactframe{14mm}{7mm}}
    \put (93,9){\makebox[14mm]{$g^{L}$}}
	\put (107,10){\line(1,0){6}}
	\put (113,10){\line(0,1){2}}

\end{picture}

\caption{Synthesis schematic for an infinitely iterated shift-invariant filter bank.} \label{SI-IFB-synthesis}
\end{figure}
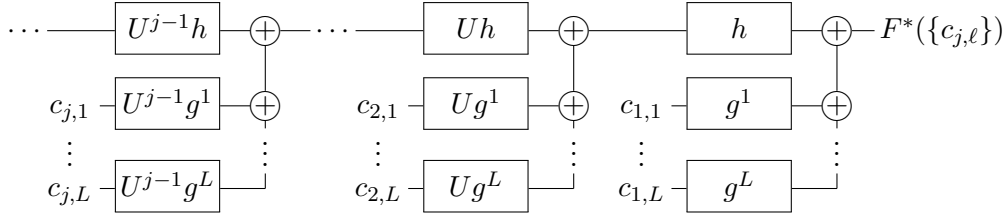

\section{Stability of Infinitely Iterated Shift-Invariant Filter Banks}

The following proposition describes a straightforward characterization of the frame bounds for an infinitely iterated shift-invariant filter bank in terms of the iterated high-pass filters.

\begin{proposition}[Characterization of Stability] \label{stable-char}
Let $h, g^{1}, g^{2}, \ldots, g^{L} \in \ell^{2}(\mathbb{Z})$ be low-pass and high-pass filters, respectively.  Fix $0<A\le B<\infty$.  The infinitely iterated shift-invariant filter bank associated with $h, g^{1}, g^{2}, \ldots, g^{L}$ is stable with frame bounds $A$ and $B$ if and only if
\begin{equation} \label{frame-char}
A\le \sum_{\ell=1}^{L} \sum_{j=1}^{\infty} \vert \hat{g}^{\ell}_{j}(\xi) \vert^{2} \le B, \quad \text{a.e.} \; \xi \in \mathbb{T}.
\end{equation}

\noindent
Moreover, the infinitely iterated shift-invariant filter bank associated with $h, g^{1}, g^{2}, \ldots, g^{L}$ is Bessel with bound $B$ if and only if the right-hand inequality of \eqref{frame-char} holds.
\end{proposition}

\begin{proof}
Let $x\in \ell^{2}(\mathbb{Z})$.  It follows from the Fubini-Tonelli theorem that
$$ \Vert Fx\Vert^{2} = \sum_{\ell=1}^{L} \sum_{j=1}^{\infty} \Vert (Fx)_{j,\ell}\Vert^{2} = \sum_{\ell=1}^{L} \sum_{j=1}^{\infty} \int_{\mathbb{T}} \vert \hat{g}_{\ell}(\xi) \vert^{2} \vert \hat{x}(\xi)\vert^{2} \; d\xi = \int_{\mathbb{T}} \left (\sum_{\ell=1}^{L} \sum_{j=1}^{\infty} \vert \hat{g}_{\ell}(\xi) \vert^{2} \right ) \vert \hat{x}(\xi)\vert^{2} \; d\xi.$$

\noindent
It follows that 
$$ A \Vert x\Vert^{2} \le \Vert Fx\Vert^{2} \le B \Vert x\Vert^{2} \quad \text{for all} \quad x\in \ell^{2}(\mathbb{Z})$$

\noindent
if and only if \eqref{frame-char} holds.  It is straightforward to observe that $\Vert F x\Vert^{2} \le B \Vert x\Vert^{2}$ for all $x\in \ell^{2}(\mathbb{Z})$ if and only if the right-hand inequality of \eqref{frame-char} holds.
\end{proof}

\begin{remark} \label{upsampledF}
Suppose that the infinitely iterated shift-invariant filter bank associated with filters $h, g^{1}, g^{2}, \ldots, g^{L}$ is stable with bounds $A$ and $B$.  Fix $j_{0}\in \mathbb{N}$ and define $\tilde{h}=U^{j_{0}}h$ and $\tilde{g}^{\ell}=U^{j_{0}}g^{\ell}$, $1\le \ell \le L$, and observe that
$$ \sum_{\ell=1}^{L} \sum_{j=1}^{\infty} \vert \hat{\tilde{g}}^{\ell}_{j} (\xi) \vert^{2} = \sum_{\ell=1}^{L} \sum_{j=1}^{\infty} \vert \hat{g}^{\ell}_{j} (2^{j_{0}}\xi) \vert^{2}.$$

\noindent
It follows that the infinitely iterated shift-invariant filter bank associated with $\tilde{h}, \tilde{g}^{1}, \tilde{g}^{2}, \ldots, \tilde{g}^{L}$ is also stable with bounds $A$ and $B$.  This fact will be useful in Section \ref{FIFBstability}.
\end{remark}

The \emph{ideal} filters given by
$$ \hat{h}(\xi) = \mathbbm{1}_{\lbrack -\frac{1}{4}, \frac{1}{4}\rbrack} (\xi) \qquad \text{and} \qquad \hat{g}(\xi) =  \mathbbm{1}_{\lbrack -\frac{1}{2},-\frac{1}{4}\rbrack \cup \lbrack \frac{1}{4}, \frac{1}{2}\rbrack} (\xi)$$

\noindent
lead to the iterated high-pass filters $\hat{g}_{j}(\xi) =  \mathbbm{1}_{\mathbb{A}_{j}}$, where
$$ \mathbb{A}_{j} = \lbrace \xi \in \mathbb{T} : 2^{-(j+1)}< \vert \xi\vert \le 2^{-j}\rbrace, \quad j \in \mathbb{N}.$$

\noindent
It is an easy consequence of Proposition \ref{stable-char} that the iterated filter bank corresponding to the ideal filters $h$ and $g$ constitutes a Parseval frame for $\ell^{2}(\mathbb{Z})$.  The downside of the ideal filters is obvious -- the sequences $h$ and $g$ are not finitely supported, making them unsuitable for use in most applications.  In a certain sense, the ideal iterated high-pass filters can be interpreted as a diagonal solution of
$$ \sum_{j=1}^{\infty} \vert \hat{g}_{j}(\xi)\vert^{2} = 1, \quad \text{a.e.} \; \xi \in \mathbb{T},$$

\noindent
with respect to the dyadic annuli $\mathbb{A}_{j}$, $\ell \in \mathbb{N}$.  Therefore, a natural generalization of interest in applications might seek diagonally dominant solutions that employ finitely supported filters $h$ and $g$.  This work focuses on the class of finitely supported low-pass filters of the form
\begin{equation} \label{haar-type}
\hat{h}(\xi) = \left \lbrack \frac{1+e^{2\pi i\xi}}{2} \right \rbrack^{n} p(\xi),
\end{equation}

\noindent
where $n\in \mathbb{N}$ and $p(\xi)$ a trigonometric polynomial satisfying $p(0)=1$.  Observe that low-pass filters of this from necessarily satisfy $\hat{h}(0)=1$ and $\hat{h}(\half)=0$.  High-pass filters considered here will also be finitely supported and will be assumed to satisfy $\hat{g}(0)=0$.

\subsection{Sufficient Condition for a Bessel Bound}

The following lemma (see \cite{BJME2022} for a proof) is adapted from many constructions of both orthonormal and biorthogonal wavelets \cite{Daubechies1992, Cohen1992, CohenDaubechiesFeauveau1992}.  The estimate provides a bound on the modulus of the iterated low-pass filter, which is crucial to the derivation of Bessel bounds for an infinitely iterated shift-invariant filter bank using a low-pass filter of the form \eqref{haar-type}.

\begin{lemma} \label{sine-lemma}
Let $J$ be a positive integer, then for $\xi \in \lbrack -\half,\half\rbrack$, 
\begin{equation} \label{sine-product}
\left \vert \prod_{k=0}^{J-1} \frac{1+e^{2\pi i 2^{k}\xi}}{2} \right \vert \le \min{\left \lbrace 1, \frac{1}{2^{J+1} \vert \xi \vert} \right \rbrace}.
\end{equation}
\end{lemma}

It is evident from \eqref{sine-product} that the cosine factor in \eqref{haar-type} leads to a natural decay for $\vert \hat{h}_{J} \vert$ on the annuli $\mathbb{A}_{j}$ with $j \le J$.  The next result focuses on the second factor, $p(\xi)$, in \eqref{haar-type}, which must be controlled adequately for the iterated filter bank to yield a Bessel bound.  This lemma is also proven in \cite{BJME2022}.

\begin{lemma} \label{p-lemma}
Let $p(\xi)$ be a trigonometric polynomial satisfying $p(0)=1$.  Assume that there exists $s\in \mathbb{N}$ and $\varepsilon >0$ such that
\begin{equation} \label{p-condition}
\sup_{\xi \in \mathbb{\mathbb{R}}} \left \vert \prod_{k=0}^{s-1}  p(2^{k}\xi) \right \vert \le 2^{(n-\varepsilon)s}.
\end{equation}

\noindent
Then, there exists $C_{1}>0$ such that for all $j,J \in \mathbb{N}$,
\begin{equation} \label{Bessel-est}
\sup_{\xi \in \mathbb{A}_{j}} \left \vert \prod_{k=0}^{J-1} p(2^{k} \xi) \right \vert \le \begin{cases} C_{1} 2^{(J-j)(n-\varepsilon)}, & 1\le j \le J, \\ C_{1}, & j > J. \end{cases} 
\end{equation}
\end{lemma}

\begin{remark}
Notice that when $s=1$, \eqref{p-condition} reduces to $\vert p(\xi) \vert \le 2^{n-\varepsilon}$.  This simple condition is sufficient for many well-known constructions, e.g., the low-pass filters studied by Burt and Adelson \cite{BurtAdelson1983}.
\end{remark}

A Bessel bound for a specific class of infinitely iterated shift-invariant filter banks is produced by combining the estimates of Lemmas \ref{sine-lemma} and \ref{p-lemma}.

\begin{theorem} \label{Bessel-bound}
Let $h\in \ell^{2}(\mathbb{Z})$ be a finitely supported low-pass filter of the form \eqref{haar-type}, where $p(\xi)$ is a trigonometric polynomial satisfying $p(0)=1$.  Let $g^{1}, g^{2}, \ldots, g^{L} \in \ell^{2}(\mathbb{Z})$ be finitely supported high-pass filters satisfying $\hat{g}^{\ell}(0)=0$, $1\le \ell \le L$.  If there exists $\varepsilon >0$ and $s\in \mathbb{N}$ such that \eqref{p-condition} holds, then
\begin{enumerate}[leftmargin=0.3275in,itemsep=0.125in,label = (\alph*)]
\item the infinitely iterated shift-invariant filter bank associated with $h, g^{1}, g^{2}, \ldots, g^{L}$ is Bessel and
\item for each $x\in \ell^{2}(\mathbb{Z})$, $\Vert x * \overline{h}_{j} \Vert \rightarrow 0$ as $j\rightarrow \infty$.
\end{enumerate}
\end{theorem}

\begin{proof}
The proof for each statement relies on estimates of the modulus of the iterated low- and high-pass filters on the dyadic annuli $\mathbb{A}_{j}$, $j \in \mathbb{N}$.  Estimates involving the high-pass filters will not depend on $\ell$ and thus it may be assumed that $L=1$ without loss of generality.  The high-pass filter will be denoted simply by $g$ rather than $g^{1}$ for simplicity.

\medskip
In light of \eqref{alpha-j}, the Fourier transform of the iterated low-pass filter of order $J-1$ is given by
$$ \hat{h}_{J-1}(\xi) = \prod_{k=0}^{J-2} \hat{h} (2^{k} \xi) = \prod_{k=0}^{J-2} \left \lbrack \frac{1+e^{2\pi i2^{k} \xi}}{2} \right \rbrack^{n}  \times \prod_{k=0}^{J-2} p(2^{k} \xi).$$

\noindent
Lemma \ref{sine-lemma} implies that
$$ \sup_{\xi \in \mathbb{A}_{j}} \prod_{k=0}^{J-2} \left \vert \frac{1+e^{2\pi i2^{k} \xi}}{2} \right \vert^{n} \le \begin{cases} 2^{(j+1-J)n}, & 1\le j \le J-1, \\ 1, & j \ge J, \end{cases} $$

\noindent
while Lemma \ref{p-lemma} provides the estimates
$$ \sup_{\xi \in \mathbb{A}_{j}} \left \vert \prod_{k=0}^{J-2} p(2^{k} \xi) \right \vert \le \begin{cases} C_{1} 2^{(J-j-1)(n-\varepsilon)}, & 1\le j \le J-1, \\ C_{1}, & j \ge J. \end{cases} $$

\noindent
Combining the two estimates leads to
\begin{equation} \label{alpha-j-est}
\sup_{\xi \in \mathbb{A}_{j}} \left \vert \hat{h}_{J-1}(\xi) \right \vert \le \begin{cases} C_{1} 2^{(j+1-J)\varepsilon}, & 1\le j \le J-1, \\ C_{1}, & j \ge J. \end{cases}
\end{equation}

\noindent
It follows from \eqref{beta-j} that the Fourier transform of the iterated high-pass filter of order $J$ is given by $\hat{g}_{J}(\xi) = \hat{g}(2^{J-1}\xi) \hat{h}_{J-1}(\xi)$.  The fact that $\hat{g}(0)=0$ leads to the estimate $\vert \hat{g}(\xi)\vert \le C_{2} \vert \xi \vert$ for $\xi \in \mathbb{T}$.  Moreover, evaluating at $2^{J-1}\xi$ leads to $\vert \hat{g}(2^{J-1}\xi)\vert \le C_{2} 2^{J-1} \vert \xi \vert$.  This estimate is useful on the annuli $\mathbb{A}_{j}$, $j \ge J$, while the simpler estimate $\vert \hat{g}(\xi)\vert \le C_{2}$ will be used for the annuli $\mathbb{A}_{j}$, $1\le j \le J-1$.  The resulting estimates for the iterated high-pass filter of order $J$ are thus
$$ \sup_{\xi \in \mathbb{A}_{j}} \vert \hat{g}_{J}(\xi)\vert \le \begin{cases} C_{1} C_{2} 2^{(j+1-J)\varepsilon}, & 1\le j \le J-1, \\ C_{1} C_{2} 2^{J-j-1}, & j \ge J. \end{cases}$$

\noindent
Letting $C_{3}=(C_{1}C_{2})^{2}$, it follows that for $\xi \in \mathbb{A}_{j}$,
\begin{align*}
\sum_{J=1}^{\infty} \vert \hat{g}_{J}(\xi)\vert^{2} &\le  \sum_{J=1}^{j} C_{3} 4^{J-j-1} + \sum_{J=j+1}^{\infty} C_{3} 4^{(j+1-J)\varepsilon} \\
&= \sum_{J=1}^{j} C_{3} 4^{-J} + \sum_{J=0}^{\infty} C_{3} \left ( 4^{\varepsilon} \right )^{-J},
\end{align*}

\noindent
which is bounded independently of $j$.  By Proposition \ref{stable-char}, the infinitely iterated shift-invariant filter bank associated with $h$ and $g$ is Bessel, completing the proof of (a).

\medskip
Now let $x\in \ell^{2}(\mathbb{Z})$ and fix $\delta >0$.  Observe that
$$ \Vert x\Vert^{2} = \sum_{j \in \mathbb{N}} \left ( \int_{\mathbb{A}_{j}} \vert \hat{x}(\xi)\vert^{2} \; d \xi \right ),$$

\noindent
so there exists $j_{0}\in \mathbb{N}$ such that 
$$ \sum_{j \ge j_{0}} \left ( \int_{\mathbb{A}_{j}} \vert \hat{x}(\xi)\vert^{2} \; d \xi \right ) < \delta.$$

\noindent
Moreover, it follows from \eqref{alpha-j-est} that $\vert \hat{h}_{J}(\xi)\vert \le C_{4} 2^{-J\varepsilon}$ for $\xi \in \mathbb{A}_{j}$, $j \le J$.  Similarly, $\vert \hat{h}_{J}(\xi)\vert \le C_{1}$ for $\xi \in \mathbb{A}_{j}$, $j \ge J$.  Therefore, assuming that $J\ge j_{0}$,
\begin{align*}
\Vert x* \overline{h}_{J}\Vert^{2} &= \int_{\mathbb{T}} \vert \hat{x}(\xi)\vert^{2} \, \vert \hat{h}_{J}(\xi)\vert^{2} \; d\xi \\
&= \sum_{j=1}^{J} \int_{\mathbb{A}_{j}} \vert \hat{x}(\xi)\vert^{2} \, \vert \hat{h}_{J}(\xi)\vert^{2} \; d\xi + \sum_{j\ge J} \int_{\mathbb{A}_{j}} \vert \hat{x}(\xi)\vert^{2} \, \vert \hat{h}_{J}(\xi)\vert^{2} \; d\xi \\
&\le \sum_{j=1}^{J} \int_{\mathbb{A}_{j}} C_{4}^{2} 4^{-J\varepsilon} \vert \hat{x}(\xi)\vert^{2} \; d\xi + \sum_{j\ge j_{0}} \int_{\mathbb{A}_{j}} C_{1}^{2} \vert \hat{x}(\xi)\vert^{2} \; d\xi \\
&\le C_{4}^{2} 4^{-J\varepsilon} \Vert x\Vert^{2} + C_{1}^{2} \delta.
\end{align*}

\noindent
The fact that $\delta >0$ was chosen arbitrarily now implies that $\Vert x*\overline{h}_{J}\Vert \rightarrow 0$ as $J\rightarrow \infty$, completing the proof of (b).
\end{proof}

It is not difficult to construct filter banks of the form \eqref{haar-type} that fail to be Bessel when $\eqref{p-condition}$ is not satisfied, as illustrated by the following example.

\begin{example}
Let $h$ be a finitely supported filter such that
$$ \hat{h}(\xi) = \half (1+e^{2\pi i \xi}) p(\xi),$$

\noindent
and assume that $p$ is an even polynomial satisfying $p(\frac{1}{6})>0$ and $p(\frac{1}{3})>2^{1+\varepsilon}$ for some $\varepsilon >0$.  Let $g$ be defined according to
$$ \hat{g}(\xi) = \overline{\hat{h}(\xi+\half)} e^{-2\pi i \xi}.$$

\noindent
It follows that
\begin{align*}
\vert \hat{h}_{j-1} (1/3) \vert &= \left \vert \prod_{k=0}^{j-2} \frac{1+e^{\frac{2^{k+1}}{3}\pi i}}{2} \right \vert \times \left \vert \prod_{k=0}^{j-2} p(2^{k}/3) \right \vert  \\
&= \left \vert \frac{\sin{(2^{j-1}\pi/3)}}{2^{j-1} \sin{(\pi/3)}} \right \vert \, \vert p(1/3)\vert^{j-1} \\
&> 2^{-(j-1)} 2^{(j-1)(1+\varepsilon)} \\
&= 2^{(j-1)\varepsilon}
\end{align*}

\noindent
and $\vert \hat{g}_{j}(1/3)\vert = \vert \hat{h}_{j-1} (1/3) \vert \, \vert \hat{g}(2^{j-1}/3) \vert > 2^{(j-1)\varepsilon} \sin{(\pi/3)} \vert p(1/6)\vert$.  Consequently, the series
$$ \sum_{j=1}^{\infty} \vert \hat{g}_{j}(1/3)\vert^{2}$$

\noindent
is divergent and for any $M>0$ there is a partial sum that will exceed $M$ on a set of positive measure.  Therefore, by Proposition \ref{stable-char}, the corresponding infinitely iterated shift-invariant filter bank is not Bessel.  The choice $p(\xi) = 2-\cos{(2\pi \xi)} = -\half e^{-2\pi i \xi} + 2 - \half e^{2\pi i \xi}$ provides a concrete example of a trigonometric polynomial $p(\xi)$ satisfying the above conditions.
\end{example}

\subsection{Sufficient Condition for a Lower Frame Bound}

The following counterpart to Theorem \ref{Bessel-bound} describes sufficient conditions for the stability of an infinitely iterated shift-invariant filter bank.

\begin{theorem} \label{frame-bound}
Let $h, g^{1}, g^{2}, \ldots, g^{L} \in \ell^{2}(\mathbb{Z})$.  Assume that $h$ is of the form \eqref{haar-type} where $p$ is a trigonometric polynomial satisfying $p(0)=1$ and the infinitely iterated shift-invariant filter bank associated with $h, g^{1}, g^{2}, \ldots, g^{L}$ is Bessel.  If there exist positive constants $a$, $\delta$, $p_{0}$, and $q_{0}$ such that
\begin{enumerate}[leftmargin=0.5in,itemsep=0.125in,label = (\roman*)]
\item $\vert p(\xi)\vert \ge p_{0}$ for $\vert \xi \vert \le \quarter$,

\item $\vert p(\xi)\vert \ge 1-a\vert \xi \vert$ for $\vert \xi \vert < \delta \le \frac{1}{2a}$, and

\item $\displaystyle{ \max_{1\le \ell \le L} \vert \hat{g}^{\ell}(\xi)\vert \ge q_{0} }$ for $\quarter \le \vert \xi \vert \le \half$,
\end{enumerate}

\noindent
then the infinitely iterated shift-invariant filter bank generated by $h, g^{1}, g^{2}, \ldots, g^{L}$ is stable.
\end{theorem}

\begin{proof}
The goal of the proof is to demonstrate a lower bound for the sum
$$ \sum_{\ell=1}^{L} \vert \hat{g}_{j}^{\ell}(\xi)\vert $$ 

\noindent
on $\mathbb{A}_{j}$ that is independent of $j$.  Recall that the Fourier transform of the iterated high-pass filter $g_{j}^{\ell}$ can be written as
$$ \hat{g}_{j}^{\ell}(\xi) = \hat{h}_{j}(\xi) \hat{g}^{\ell}(2^{j-1}\xi) = \left \lbrack \prod_{k=0}^{j-2} \left ( \frac{1+e^{2\pi i 2^{k} \xi}}{2} \right )^{n} p(2^{k}\xi) \right \rbrack \hat{g}^{\ell}(2^{j-1}\xi).$$

\noindent
It will be convenient to first estimate
$$ \left \vert \prod_{k=0}^{j-2} \frac{1+e^{2\pi i 2^{k} \xi}}{2} \right \vert = \frac{\sin{(2^{j-1}\pi \xi)}}{2^{j-1} \vert \sin{(\pi \xi)}\vert }$$

\noindent
on the annulus $\mathbb{A}_{j}$.  Notice that if $\xi \in \mathbb{A}_{j}$, then $2^{j-1}\vert \xi\vert \le \half$ and the secant approximation of $\sin{(\pi \xi)}$ on $\lbrack 0,\half\rbrack$ leads to the estimate $\vert \sin{(\pi 2^{j-1}\xi)}\vert \ge 2^{j} \vert \xi \vert$.  Combining this estimate with the fact that $\vert \sin {(\pi \xi)}\vert \le \pi \vert \xi \vert$ leads to
$$\left \vert \prod_{k=0}^{j-2} \frac{1+e^{2\pi i 2^{k} \xi}}{2} \right \vert \ge \frac{2^{j}\vert \xi \vert}{2^{j-1} \pi \vert \xi \vert} \ge \frac{2}{\pi}, \quad \xi \in \mathbb{A}_{j}.$$

\noindent
It follows that
\begin{equation} \label{est-1}
 \left \vert \prod_{k=0}^{j-2} \frac{1+e^{2\pi i 2^{k} \xi}}{2} \right \vert^{n} \ge  \frac{2^{n}}{\pi^{n}}, \quad \xi \in \mathbb{A}_{j}.
\end{equation}

\noindent
The next step is to derive a lower bound on the dyadic product of the terms $\vert p(2^{k}\xi)\vert$ for $\xi \in \mathbb{A}_{j}$.  If $\xi \in \lbrack -\delta, \delta \rbrack$, then since $\delta \le \frac{1}{2a}$ it follows that
$$ \prod_{k=0}^{j-2} \vert p(2^{-k} \xi) \vert \ge \prod_{k=0}^{j-2} (1-2^{-k}a\vert \xi \vert) \ge \prod_{k=0}^{j-2} (1-2^{-(k+1)}) \ge \prod_{k=0}^{\infty} (1-2^{-(k+1)}) =: p_{\infty}>0.$$

\noindent
Suppose now that $\xi \in \mathbb{A}_{j}$.  The points $2^{k}\xi$, $0\le k \le j-2$ all lie within $\lbrack -\quarter,\quarter \rbrack$ and no more than $K=\lfloor \log_{2}\delta^{-1} \rfloor$ of these points can lie outside the interval $\lbrack -\delta, \delta \rbrack$.  Therefore, it follows that
\begin{equation} \label{est-2}
\left ( \prod_{k=0}^{j-2} \vert p(2^{k} \xi)\vert \right ) \ge p_{\infty} p_{0}^{K} > 0.
\end{equation}

\noindent
Finally, observe that for $\xi \in \mathbb{A}_{j}$, it follows that $2^{j-1}\vert \xi\vert \in \lbrack \quarter, \half \rbrack$ which implies that
$$ \max_{1\le \ell \le L} \vert \hat{g}^{\ell}(2^{j-1}\xi) \vert \ge q_{0}.$$

\noindent
Bringing this observation together with the estimates \eqref{est-1} and \eqref{est-2}, it follows that
$$\max_{1\le \ell \le L} \vert \hat{g}_{j}^{\ell}(\xi)\vert \ge \frac{2^{n} p_{\infty} p_{0}^{K} q_{0}}{\pi^{n}} > 0$$

\noindent
for all $\xi \in \mathbb{A}_{j}$.  In light of Proposition \ref{stable-char}, the infinitely iterated shift-invariant filter bank associated with $h, g^{1}, g^{2}, \ldots, g^{L}$ is stable.
\end{proof}

\section{Stability of Finitely Iterated Shift-Invariant Filter Banks} \label{FIFBstability}

This section focuses on the correspondence of stability between infinitely iterated shift-invariant filter banks and their finitely iterated counterparts.  The analysis schematic for the \emph{finitely iterated dyadic filter bank of order $J$} associated with low-pass filter $h$ and a high-pass filters $g^{1}, g^{2}, \ldots, g^{L}$ is depicted in Figure \ref{SI-FIFB-analysis}.  The \emph{filter bank analysis operator of order $J$} associated with the filters $h, g^{1}, g^{2}, \ldots, g^{L} \in \ell^{2}(\mathbb{Z})$ and acting on $x\in \ell^{2}(\mathbb{Z})$ is the mapping 
$$F_{J}: x \mapsto F_{J} x := \lbrace (F_{J} x)_{j,\ell} \rbrace_{1\le \ell \le L, 1\le j \le J} \cup \lbrace (F_{J}x)_{J+1} \rbrace,$$

\noindent
where $(F_{J} x)_{j,\ell} = x*\overline{g}^{\ell}_{j}$ for $1\le j \le J$ and $(F_{J} x)_{J+1} = x*\overline{h}_{J}$.  Notice that $(F_{J} x)_{j,\ell} = (Fx)_{j,\ell}$ for $1\le j \le J$, while $(F_{J} x)_{J+1}$ accounts for the contribution of the iterated low-pass filter of order $J$.  

\begin{definition} \label{finite-FB-frame}
The finitely iterated shift-invariant filter bank is said to be \emph{stable} when there exist constants $0<A\le B <\infty$ such that for all $x\in \ell^{2}(\mathbb{Z})$,
$$ A\Vert x\Vert^{2} \le \Vert (F_{J} x)_{J+1} \Vert^{2} + \sum_{\ell=1}^{L} \sum_{j=1}^{J} \Vert (F_{J} x)_{j,\ell} \Vert^{2} \le B \Vert x \Vert^{2}.$$
\end{definition}

\noindent
The terms \emph{Bessel} and \emph{Parseval} will be applied to finitely iterated shift-invariant filter banks in the same manner that they are used in the infinitely iterated context.

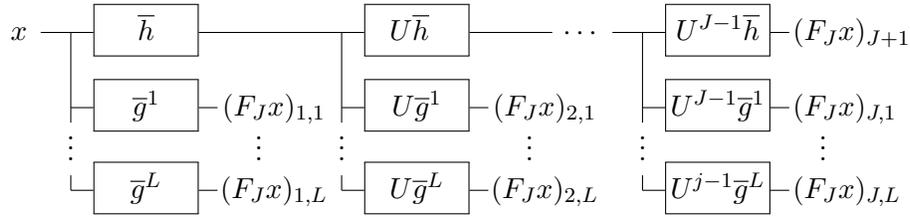
\begin{figure}[hbtp]
\centering
\begin{picture}(120,27)(1,8)
	\put (1,30){$x$}
	\put (5,31){\line(1,0){7}}
	\put (9,31){\line(0,-1){12}}
	\put (12,29){\exactframe{14mm}{7mm}}
    \put (12,30){\makebox[14mm]{$\overline{h}$}}
	\put (26,31){\line(1,0){22}}
	\put (45,31){\line(0,-1){12}}
	\put (48,29){\exactframe{14mm}{7mm}}
    \put (47,30){\makebox[14mm]{$U\overline{h}$}}
	\put (62,31){\line(1,0){11}}
    \put (74,30){\makebox[6mm]{$\cdots$}}
	\put (81,31){\line(1,0){7}}
	\put (85,31){\line(0,-1){12}}
	\put (88,29){\exactframe{14mm}{7mm}}
    \put (88,30){\makebox[14mm]{$U^{J-1}\overline{h}$}}
	\put (102,31){\line(1,0){2.5}}
	\put (105,30){$(F_{J} x)_{J+1}$}

	\put (9,21){\line(1,0){3}}
	\put (12,19){\exactframe{14mm}{7mm}}
    \put (12,20){\makebox[14mm]{$\overline{g}^{1}$}}
	\put (26,21){\line(1,0){2.5}}
	\put (32,20){\makebox[8mm]{$(F_{J} x)_{1,1}$}}
	\put (45,21){\line(1,0){3}}
	\put (48,19){\exactframe{14mm}{7mm}}
    \put (48,20){\makebox[14mm]{$ U\overline{g}^{1}$}}
	\put (62,21){\line(1,0){2.5}}
	\put (68,20){\makebox[8mm]{$(F_{J} x)_{2,1}$}}
	\put (85,21){\line(1,0){3}}
	\put (88,19){\exactframe{14mm}{7mm}}
    \put (88,20){\makebox[14mm]{$U^{J-1}\overline{g}^{1}$}}
	\put (102,21){\line(1,0){2.5}}
	\put (105,20){$(F_{J} x)_{J,1}$}

    \put (8,14){\makebox[2mm]{$\vdots$}}
    \put (33,14){\makebox[2mm]{$\vdots$}}
    \put (44,14){\makebox[2mm]{$\vdots$}}
    \put (69,14){\makebox[2mm]{$\vdots$}}
    \put (84,14){\makebox[2mm]{$\vdots$}}
    \put (108,14){\makebox[2mm]{$\vdots$}}

	\put (9,10){\line(0,1){2}}
	\put (9,10){\line(1,0){3}}
	\put (12,8){\exactframe{14mm}{7mm}}
    \put (12,9){\makebox[14mm]{$\overline{g}^{L}$}}
	\put (26,10){\line(1,0){2.5}}
	\put (32,9){\makebox[8mm]{$(F_{J} x)_{1,L}$}}
	\put (45,10){\line(0,1){2}}
	\put (45,10){\line(1,0){3}}
	\put (48,8){\exactframe{14mm}{7mm}}
    \put (48,9){\makebox[14mm]{$ U\overline{g}^{L}$}}
	\put (62,10){\line(1,0){2.5}}
	\put (68,9){\makebox[8mm]{$(F_{J} x)_{2,L}$}}
	\put (85,10){\line(0,1){2}}
	\put (85,10){\line(1,0){3}}
	\put (88,8){\exactframe{14mm}{7mm}}
    \put (88,9){\makebox[14mm]{$U^{j-1}\overline{g}^{L}$}}
	\put (102,10){\line(1,0){2.5}}
	\put (105,9){$(F_{J} x)_{J,L}$}

\end{picture}

\caption{Analysis schematic for a finitely iterated shift-invariant filter bank.} \label{SI-FIFB-analysis}
\end{figure}

\begin{theorem} \label{infinite2finite}
Let $h, g^{1}, g^{2}, \ldots, g^{L} \in \ell^{2}(\mathbb{Z})$ be low- and high-pass filters.  If the infinitely iterated shift-invariant filter bank associated with $h, g^{1}, g^{2}, \ldots, g^{L}$ is stable with bounds $A$ and $B$, then the finitely iterated shift-invariant filter bank of order $J$ is stable for each $J\in \mathbb{N}$ with bounds $\min{\lbrace A, A/B\rbrace}$ and $\max{\lbrace B/A, B\rbrace}$.
\end{theorem}

\begin{proof}
Throughout the proof, the quantities $x*\overline{h}_{j}$ and $x*\overline{g}^{\ell}_{j}$, $j\in \mathbb{Z}$, will be denoted by $x_{j}$ and $y^{\ell}_{j}$, respectively.  Fix $x\in \ell^{2}(\mathbb{Z})$.  The infinitely iterated shift-invariant filter bank associated with $h$ and $g$ is stable with bounds $A$ and $B$, which implies that
$$ A \Vert x\Vert^{2} \le \sum_{\ell=1}^{L} \sum_{j=1}^{\infty} \Vert y^{\ell}_{j}\Vert^{2} \le B \Vert x\Vert^{2}.$$
 
\noindent
The goal is to establish lower and upper bounds on $\Vert x_{J}\Vert^{2} + \sum_{\ell=1}^{L} \sum_{j=1}^{J} \Vert y_{j}^{\ell}\Vert^{2}$ that are independent of $J$.  The argument relies on the fact that $y_{j}^{\ell}$, $j \ge J$, can be computed from $x_{J}$ using the infinitely iterated filter bank associated with the filters $U^{J}h$ and $U^{J}g^{\ell}$, as shown in Figure \ref{proof-figure}.  Remark \ref{upsampledF} establishes the fact that the upsampled filter bank is also stable with bounds $A$ and $B$ for any $J\in \mathbb{N}$.  It will be convenient to denote the analysis operator for the $J$-upsampled filter bank by $F^{[J]}$.  Three cases will be considered.

\begin{figure}[hbtp]
\centering
\begin{picture}(120,28)(1,6)
	\put (0,30){$x_{J}$}
	\put (5,31){\line(1,0){7}}

	\put (9,31){\line(0,-1){13}}
	\put (8,13.5){\makebox[2mm]{$\vdots$}}
    \put (12,29){\exactframe{14mm}{7mm}}
    \put (12,30){\makebox[14mm]{$U^{J}\overline{h}$}}
	\put (26,31){\line(1,0){11}}
	\put (37,31){\line(1,0){8}}

	\put (42,31){\line(0,-1){13}}
	\put (41,13.5){\makebox[2mm]{$\vdots$}}
    \put (45,29){\exactframe{14mm}{7mm}}
    \put (45,30){\makebox[14mm]{$U^{J+1}\overline{h}$}}
	\put (59,31){\line(1,0){19}}
	
    \put (9,13){\line(0,-1){5}}
    \put (9,21){\line(1,0){3}}
    \put (12,19){\exactframe{14mm}{7mm}}
    \put (12,20){\makebox[14mm]{$U^{J}\overline{g}^{1}$}}
	\put (26,21){\line(1,0){4}}
	\put (30.5,20){\makebox[8mm]{$y_{J+1}^{1}$}}
	\put (33.5,13.5){\makebox[2mm]{$\vdots$}}
	
	\put (42,13){\line(0,-1){5}}
	\put (42,21){\line(1,0){3}}
    \put (45,19){\exactframe{14mm}{7mm}}
    \put (45,20){\makebox[14mm]{$ U^{J+1}\overline{g}^{1}$}}
	\put (59,21){\line(1,0){4}}
	\put (63.5,20){\makebox[8mm]{$y_{J+2}^{1}$}}
	\put (66.5,13.5){\makebox[2mm]{$\vdots$}}

	\put (9,8){\line(1,0){3}}
    \put (12,6){\exactframe{14mm}{7mm}}
    \put (12,7){\makebox[14mm]{$U^{J}\overline{g}^{L}$}}
	\put (26,8){\line(1,0){4}}
	\put (30.5,7){\makebox[8mm]{$y_{J+1}^{L}$}}
	
	\put (42,8){\line(1,0){3}}
    \put (45,6){\exactframe{14mm}{7mm}}
    \put (45,7){\makebox[14mm]{$U^{J+1}\overline{g}^{L}$}}
	\put (59,8){\line(1,0){4}}
	\put (63.5,7){\makebox[8mm]{$y_{J+2}^{L}$}}

	\put (75,31){\line(0,-1){13}}
	\put (74,13.5){\makebox[2mm]{$\vdots$}}
    \put (78,29){\exactframe{14mm}{7mm}}
    \put (78,30){\makebox[14mm]{$U^{J+2}\overline{h}$}}
    \put (92,31){\line(1,0){12}}
	\put (105,30){\makebox[6mm]{$\cdots$}}

    \put (75,13){\line(0,-1){5}}
    \put (75,21){\line(1,0){3}}
    \put (78,19){\exactframe{14mm}{7mm}}
    \put (78,20){\makebox[14mm]{$U^{J+2}\overline{g}^{1}$}}
	\put (92,21){\line(1,0){4}}
	\put (97.5,20){\makebox[8mm]{$y_{J+3}^{1}$}}
	\put (100.5,13.5){\makebox[2mm]{$\vdots$}}

	\put (75,8){\line(1,0){3}}
    \put (78,6){\exactframe{14mm}{7mm}}
    \put (78,7){\makebox[14mm]{$U^{J+2}\overline{g}^{L}$}}
	\put (92,8){\line(1,0){4}}
	\put (97.5,7){\makebox[8mm]{$y_{J+3}^{L}$}}
\end{picture}

\caption{Analysis schematic for a $J$-upsampled infinitely iterated shift-invariant filter bank.} \label{proof-figure}
\end{figure}
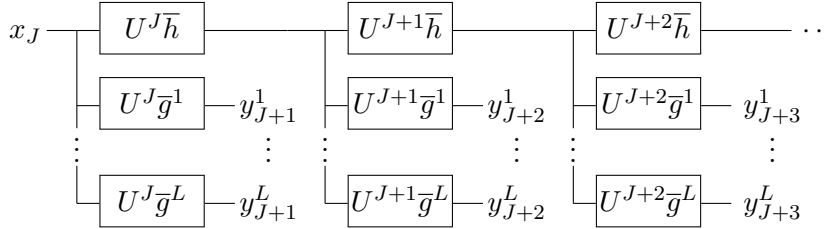

\medskip \noindent
\underline{Case 1:} Suppose that $A\le 1 \le B$.  The above observations lead to the identities
$$\Vert Fx\Vert^{2} = \sum_{\ell=1}^{L} \sum_{j\in \mathbb{N}} \Vert y_{j}^{\ell}\Vert^{2} = \sum_{\ell=1}^{L} \sum_{j=1}^{J} \Vert y_{j}^{\ell}\Vert^{2} + \Vert F^{[J]} x_{J} \Vert^{2}$$

\noindent
as well as the inequalities $A\Vert x\Vert^{2} \le \Vert Fx\Vert^{2} \le B\Vert x\Vert^{2}$ and $A \Vert x_{J} \Vert^{2} \le \Vert F^{[J]} x_{J}\Vert^{2} \le B \Vert x_{J}\Vert^{2}$.  It follows that
$$ \Vert x_{J}\Vert^{2} + \sum_{\ell=1}^{L} \sum_{j=1}^{J} \Vert y_{j}^{\ell}\Vert^{2} \le \frac{1}{A} \Vert F^{[J]} x_{J}\Vert^{2} + \sum_{\ell=1}^{L} \sum_{j=1}^{J} \Vert y_{j}^{\ell} \Vert^{2} \le \frac{1}{A} \sum_{j \in \mathbb{N}} \Vert y_{j}^{\ell}\Vert^{2} \le \frac{B}{A} \Vert x\Vert^{2}.$$

\noindent
Similarly, observe that
\begin{align*}
\Vert x_{J}\Vert^{2} + \sum_{\ell=1}^{L} \sum_{j=1}^{J} \Vert y_{j}^{\ell}\Vert^{2} &\ge \frac{1}{B} \Vert F^{[J]} x_{J}\Vert^{2} + \sum_{\ell=1}^{L} \sum_{j=1}^{J} \Vert y_{j}^{\ell} \Vert^{2} \\
&= \frac{1}{B} \left ( \Vert F x\Vert^{2} - \sum_{\ell=1}^{L} \sum_{j=1}^{J} \Vert y_{j}^{\ell} \Vert^{2} \right ) + \sum_{\ell=1}^{L} \sum_{j=1}^{J} \Vert y_{j}^{\ell} \Vert^{2} \\
&\ge \frac{A}{B} \Vert x\Vert^{2}.
\end{align*}

\noindent
Therefore, the finitely iterated shift-invariant filter bank associated with $h$ and $g$ is stable with bounds $A/B$ and $B/A$.

\medskip \noindent
\underline{Case 2:} Suppose that $1\le A\le B$.  Define $\tilde{g}^{\ell}=g^{\ell}/\sqrt{A}$ and let $\tilde{F}$ represent the analysis operator of the infinitely iterated shift-invariant filter bank associated with $h$ and $\tilde{g}^{\ell}$, $1\le \ell \le L$.  Similarly, denote $x*\overline{\tilde{g}}_{j}^{\ell}$ by $\tilde{y}_{j}^{\ell}$.  By construction, the infinitely iterated shift-invariant filter bank associated with $h$ and $\tilde{g}^{\ell}$, $1\le \ell \le L$, is stable with bounds $1$ and $B/A$.  Observe that
$$ \Vert x_{J}\Vert^{2} + \sum_{\ell=1}^{L} \sum_{j=1}^{J} \Vert y_{j}^{\ell}\Vert^{2} \le A \left ( \Vert x_{J}\Vert^{2} + \sum_{\ell=1}^{L} \sum_{j=1}^{J} \Vert \tilde{y}_{j}^{\ell}\Vert^{2} \right ) \le B \Vert x\Vert^{2}$$

\noindent
using the upper bound provided by Case 1.  Similarly, the lower bound for Case 1 leads to
$$ \Vert x_{J}\Vert^{2} + \sum_{\ell=1}^{L} \sum_{j=1}^{J} \Vert y_{j}^{\ell}\Vert^{2} \ge  \Vert x_{J}\Vert^{2} + \sum_{\ell=1}^{L} \sum_{j=1}^{J} \Vert \tilde{y}_{j}^{\ell}\Vert^{2} \ge \frac{A}{B} \Vert x\Vert^{2},$$

\noindent
showing that the finitely iterated shift-invariant filter bank associated with $h$ and $g$ is stable with bounds $A/B$ and $B$.

\medskip \noindent
\underline{Case 3:} Suppose that $A\le B\le 1$.  Define $\tilde{g}^{\ell}=g^{\ell}/\sqrt{B}$ and let $\tilde{F}$ now represent the analysis operator of the infinitely iterated shift-invariant filter bank associated with $h$ and $\tilde{g}^{\ell}$, $1\le \ell \le L$.  As before, denote $x*\overline{\tilde{g}}_{j}^{\ell}$ by $\tilde{y}_{j}^{\ell}$.  By construction, the infinitely iterated shift-invariant filter bank associated with $h$ and $\tilde{g}^{\ell}$, $1\le \ell \le L$, is stable with bounds $A/B$ and $1$.  Observe that
$$ \Vert x_{J}\Vert^{2} + \sum_{\ell=1}^{L} \sum_{j=1}^{J} \Vert y_{j}^{\ell}\Vert^{2} \le \Vert x_{J}\Vert^{2} + \sum_{\ell=1}^{L} \sum_{j=1}^{J} \Vert \tilde{y}_{j}^{\ell}\Vert^{2} \le \frac{B}{A} \Vert x\Vert^{2}$$

\noindent
by Case 1.  Similarly, 
$$ \Vert x_{J}\Vert^{2} + \sum_{\ell=1}^{L} \sum_{j=1}^{J} \Vert y_{j}^{\ell}\Vert^{2} \ge  B \left ( \Vert x_{J}\Vert^{2} + \sum_{\ell=1}^{L} \sum_{j=1}^{J} \Vert \tilde{y}_{j}^{\ell}\Vert^{2} \right ) \ge  A \Vert x\Vert^{2},$$

\noindent
showing that the finitely iterated shift-invariant filter bank associated with $h$ and $g$ is stable with bounds $A$ and $B/A$.

\medskip
The claimed stability bounds for the finitely iterated shift-invariant filter bank incorporate the findings of the three separate cases.
\end{proof}

It is natural to consider whether the converse to Theorem \ref{infinite2finite} holds, which would guarantee the stability of an infinitely iterated shift-invariant filter bank under the assumption that the finitely iterated counterparts are stable with uniform bounds.  The next proposition identifies a necessary condition on the low-pass filter for this to occur.

\begin{proposition} \label{stableZERO}
Let $h, g^{1}, g^{2}, \ldots g^{L} \in \ell^{2}(\mathbb{Z})$.  If the infinitely iterated shift-invariant filter bank associated with $h, g^{1}, g^{2}, \ldots, g^{L}$ is stable with bounds $A$ and $B$, then, for each $x\in \ell^{2}(\mathbb{Z})$, $\Vert x * \overline{h}_{J}\Vert \rightarrow 0$ as $J\rightarrow \infty$.
\end{proposition}

\begin{proof}
Let $x\in \ell^{2}(\mathbb{Z})$.  Let $x_{j}$, $y_{j}^{\ell}$, and $F^{[J]}$ be defined as in the proof of Theorem \ref{infinite2finite}.  It suffices to prove that $\Vert x_{j}\Vert \rightarrow 0$ as $j\rightarrow \infty$.  The stability of the infinitely iterated shift-invariant filter bank implies that $A\Vert x\Vert^{2} \le \Vert Fx\Vert^{2} \le B \Vert x\Vert^{2}$ and, by Remark \ref{upsampledF}, it follows that
$$ A\Vert x_{J}\Vert \le \Vert F^{[J]} x_{J} \Vert^{2} \le B\Vert x_{J}\Vert^{2}.$$

\noindent
Therefore, since $\Vert F x\Vert^{2} = \Vert F^{[J]} x_{J}\Vert^{2} + \sum_{\ell=1}^{L} \sum_{j=1}^{J} \Vert y_{j}^{\ell}\Vert^{2}$, it follows that
$$ \Vert x_{J} \Vert^{2} \le \frac{1}{A} \Vert F^{[J]} x_{J}\Vert^{2} = \frac{1}{A} \left ( \Vert Fx\Vert^{2} - \sum_{\ell=1}^{L} \sum_{j=1}^{J} \Vert y_{j}^{\ell}\Vert^{2} \right ) = \frac{1}{A} \sum_{j\ge J+1} \Vert y_{j}^{\ell}\Vert^{2},$$

\noindent
which tends to zero as $J\rightarrow \infty$ since $\sum_{\ell = 1}^{L} \sum_{j \in \mathbb{N}} \Vert y_{j}^{\ell}\Vert^{2} \le B\Vert x\Vert^{2}$.
\end{proof}

The following result describes a partial converse to Theorem \ref{infinite2finite} under the additional assumption on the low-pass filter identified by Proposition \ref{stableZERO}.  It is unclear to the authors whether or not this extra assumption is strictly necessary.

\begin{proposition} \label{finite2infinite}
Let $h, g^{1}, g^{2}, \ldots g^{L} \in \ell^{2}(\mathbb{Z})$.  If the finitely iterated shift-invariant filter bank of order $J$ associated with $h, g^{1}, g^{2}, \ldots, g^{L}$ is stable for each $J\in \mathbb{N}$ with bounds $A$ and $B$ and, for each $x\in \ell^{2}(\mathbb{Z})$, $\Vert x * \overline{h}_{J}\Vert \rightarrow 0$ as $J\rightarrow \infty$, then the infinitely iterated shift-invariant filter bank associated with $h, g^{1}, g^{2}, \ldots, g^{L}$ is stable with bounds $A$ and $B$.
\end{proposition}

\begin{proof}
Let $x\in \ell^{2}(\mathbb{Z})$.  Let $x_{j}$ and $y_{j}^{\ell}$ be defined as in the proof of Theorem \ref{infinite2finite}.  By hypothesis, 
$$ A\Vert x\Vert^{2} \le \Vert x_{J} \Vert^{2} + \sum_{\ell=1}^{L} \sum_{j=1}^{J} \Vert y_{j}^{\ell} \Vert^{2} \le B \Vert x\Vert^{2}$$

\noindent
holds for each $J\in \mathbb{N}$.  It follows immediately that
$$ \sum_{\ell=1}^{L} \sum_{j \in \mathbb{N}} \Vert y_{j}^{\ell}\Vert^{2} \le B \Vert x\Vert^{2}.$$

\noindent
Moreover, for each $J\in \mathbb{N}$, one has
$$ \sum_{\ell=1}^{L} \sum_{j=1}^{J} \Vert y_{j}^{\ell}\Vert^{2} \ge A \Vert x\Vert^{2} - \Vert x_{J}\Vert^{2}.$$

\noindent
Since $\Vert x_{J}\Vert\rightarrow 0$ as $J\rightarrow \infty$, it follows that
$$ \sum_{\ell=1}^{L} \sum_{j \in \mathbb{N}} \Vert y_{j}^{\ell}\Vert^{2} \ge A \Vert x\Vert^{2}.$$
\end{proof}

It will be useful in later sections to establish a perfect reconstruction criterion for filters associated with a shift-invariant filter bank, which guarantees that any finitely iterated shift-invariant filter bank will be Parseval.

\begin{proposition} \label{PerfectRecon}
Let $h, g^{1}, g^{2}, \ldots, g^{L} \in \ell^{2}(\mathbb{Z})$.  Assume that $h$ is a low-pass filter and that $g^{1}, g^{2}, \ldots, g^{L}$ are high-pass filters.  If
\begin{equation} \label{PReq}
\vert \hat{h}(\xi)\vert^{2} + \sum_{\ell=1}^{L} \vert \hat{g}^{\ell}(\xi) \vert^{2} = 1, \; \text{a.e.} \; \xi \in \mathbb{T},
\end{equation}

\noindent
then the finitely iterated shift-invariant filter bank of order $J$ associated with $h, g^{1}, g^{2}, \ldots, g^{L}$ is Parseval for each $J\in \mathbb{N}$.
\end{proposition}

\begin{proof}
Fix $x\in \ell^{2}(\mathbb{Z})$.  Let $x_{j}$ and $y_{j}^{\ell}$ be defined as in the proof of Theorem \ref{infinite2finite}.  Notice that 
$$ \Vert F_{1} x\Vert^{2} = \Vert x_{1}\Vert^{2} + \sum_{\ell=1}^{L} \Vert y_{1}^{\ell}\Vert^{2} = \int_{\mathbb{T}} \vert \hat{x}(\xi) \vert^{2} \left \lbrack \vert \hat{h}(\xi)\vert^{2} + \sum_{\ell=1}^{L} \vert \hat{g}^{\ell}(\xi)\vert^{2} \right \rbrack \; d\xi.$$

\noindent
It follows immediately from \eqref{PReq} that the finitely iterated shift-invariant filter bank of order $1$ is Parseval.  Observe that
$$ \vert \hat{h}_{J+1}(\xi)\vert^{2} + \sum_{\ell=1}^{L} \vert \hat{g}_{J+1}^{\ell}(\xi)\vert^{2} = \vert \hat{h}_{J}(\xi)\vert^{2} \left \lbrack \vert \hat{h}(2^{J}\xi)\vert^{2} + \sum_{\ell=1}^{L} \vert \hat{g}^{\ell}(2^{J}\xi)\vert^{2} \right \rbrack = \vert \hat{h}_{J}(\xi)\vert^{2} $$

\noindent
for a.e. $\xi \in \mathbb{T}$ by \eqref{PReq}.  It is now straightforward to show that 
$$ \Vert x_{J+1}\Vert^{2} + \sum_{\ell=1}^{L} \Vert y_{J+1}^{\ell}\Vert^{2} = \Vert x_{J}\Vert^{2}$$

\noindent
so that $\Vert F_{J+1}x\Vert^{2} = \Vert F_{J}x\Vert^{2}$.  Thus, by induction, the finitely iterated shift-invariant filter bank of order $J$ is Parseval for each $J\in \mathbb{N}$.
\end{proof}

The following counterpart to Proposition \ref{stable-char} will be used in the Section \ref{separable}.  Its proof is nearly identical to that of Proposition \ref{stable-char} and will be omitted.

\begin{proposition} \label{stable-finite}
Let $h, g^{1}, g^{2}, \ldots, g^{L} \in \ell^{2}(\mathbb{Z})$.  Assume that $h$ is a low-pass filter and $g^{1}, g^{2}, \ldots, g^{L}$ are high-pass filters.  Fix $0<A\le B<\infty$.  The finitely iterated shift-invariant filter bank of order $J\in \mathbb{N}$ associated with $h, g^{1}, g^{2}, \ldots, g^{L}$ is stable with frame bounds $A$ and $B$ if and only if
\begin{equation} \label{frame-finite}
A\le \vert \hat{h}_{J}(\xi)\vert^{2} +  \sum_{\ell=1}^{L} \sum_{j=1}^{J} \vert \hat{g}^{\ell}_{j}(\xi) \vert^{2} \le B, \quad \text{a.e.} \; \xi \in \mathbb{T}.
\end{equation}

\noindent
Moreover, the finitely iterated shift-invariant filter bank of order $J$ associated with $h, g^{1}, g^{2}, \ldots, g^{L}$ is Bessel with bound $B$ if and only if the right-hand inequality of \eqref{frame-finite} holds.
\end{proposition}

\section{Examples}

Several examples will be presented in order to illustrate the usefulness of Theorems \ref{Bessel-bound} and \ref{frame-bound} for the construction of shift-invariant filter banks that are stable under infinitely many iterations.  In each example, The time frequency localization of the filters will be provided.  The time frequency localization of the filters will be described in terms of the quantities $\sigma_{n}^{2}$ and $\sigma_{\omega}^{2}$ as used by Haddad et al \cite{Haddad:etal1993} and, more recently, by Bhati et al \cite{BatiPachoriSharmaGadre2018}.  The time spread of $x\in \ell^{2}(\mathbb{Z})$ is given by
$$ \sigma_{n}^{2} = \frac{1}{\Vert x\Vert^{2}} \sum_{n\in \mathbb{Z}} (n-n_{0})^{2} \vert x(n)\vert^{2},$$

\noindent
where
$$ n_{0} = \frac{1}{\Vert x\Vert^{2}} \sum_{n\in \mathbb{Z}} n \vert x(n)\vert^{2}.$$

\noindent
The frequency spread of $x\in \ell^{2}(\mathbb{Z})$ is given by
$$ \sigma_{\omega}^{2} = \frac{(2\pi)^{2}}{\Vert x\Vert^{2}} \int_{-\frac{1}{2}}^{\frac{1}{2}} \vert \xi \hat{x}(\xi)\vert^{2} \; d\xi,$$

\noindent
where
$$ \omega_{0} = \frac{2\pi}{\Vert x\Vert^{2}} \int_{-\frac{1}{2}}^{\frac{1}{2}} \xi \vert \hat{x}(\xi)\vert^{2} \; d\xi.$$

\noindent
The definition of $\sigma_{\omega}^{2}$ requires modification for band-pass filters.  The reader is referred to the work of Haddad et al for the details \cite{Haddad:etal1993}. 

\subsection{Symmetric Filter Design}

Symmetric filters are often desired in applications for a variety of reasons, including the fact that such filters have linear phase.  However, as in the case of the usual discrete wavelet transform \cite[Theorem 8.1.4]{Daubechies1992}, the perfect reconstruction equation \eqref{PReq} is not consistent with symmetric, finitely supported filters.  To see this, assume that the filters $h$ and $g$, respectively, have the form
$$ \hat{h}(\xi) = \sum_{k=-n}^{n} h(k) e^{-2\pi i k \xi} \qquad \text{and} \qquad \hat{g}(\xi) = \sum_{k=-n}^{n} g(k) e^{-2\pi i k\xi},$$

\noindent
where $h(k), g(k) \in \mathbb{R}$ satisfy $h(-k)=h(k)$ and $g(-k)=g(k)$.  It is sufficient to consider the case in which the filter coefficients are symmetric about zero, since multiplication by $e^{2\pi i \xi}$ has no effect on \eqref{PReq}.  Moreover, it can be assumed, Without loss of generality, that $h(n)^{2}+g(n)^{2}>0$.  It follows that the expression $\vert \hat{h}(\xi)\vert^{2} + \vert \hat{g}(\xi)\vert^{2}$ can be written in the form
$$ C_{0} + C_{1} \cos{(2\pi \xi)} + \cdots + C_{2n} \cos{(4n\pi \xi)}$$

\noindent
with $C_{2n} = 2h(n)^{2}+2g(n)^{2}>0$.  It is thus impossible for \eqref{PReq} to hold, since this would require $C_{2n}=0$.  It is natural, then, to seek such symmetric filter designs for which the coefficients $C_{k}$, $k\ge 1$, are small.

\begin{example} \label{eg5.1}
As a first example of a symmetric filter design, consider a finitely supported low-pass filter $h$ of the form \eqref{haar-type} with $n=2$ and 
$$p(\xi) = (1+a) - a \cos{(2\pi \xi)},$$

\noindent
where $a$ is a positive constant.  In order that $h$ satisfy \eqref{p-condition} with $s=1$ the constant $a$ must be chosen so that $a< 1.5$.  If $a<1.5$ and $g$ is a finitely supported high-pass filter such that $\vert \hat{g}(\xi)\vert >0$ for $\frac{1}{4} \le \vert \xi \vert \le \frac{1}{2}$, then it follows from Theorems \ref{Bessel-bound} and \ref{frame-bound} that the infinitely iterated shift-invariant filter bank associated with $h$ and $g$ is stable.  Theorem \ref{infinite2finite} guarantees that the finitely iterated shift-invariant filter bank of order $J$ associated with $h$ and $g$ will be stable for each $J\in \mathbb{N}$, with uniform bounds.  One way to guarantee that the required conditions on $g$ are satisfied is to follow the usual construction of the high-pass filter for an orthonormal wavelet and choose $g$ according to $g(n) = (-1)^{1-k} h(1-k)$.  A numerical study was made in order choose the parameter $a$ to obtain small variation between the minimum and maximum values of $\vert \hat{h}(\xi)\vert^{2} + \vert \hat{g}(\xi)\vert^{2}$ varies, leading to an estimate of $a=0.410013$.  The coefficients of $h$ and $g$ for the corresponding filters are presented in Table \ref{table-5.1}.  Both the low- and high-pass filters yield $\sigma_{n}^{2}\approx 0.296$ and $\sigma_{\omega}^{2}\approx 1.08$, leading to the time-frequency product $\sigma_{n}^{2} \sigma_{\omega}^{2} \approx 0.320$.  The frequency response of the filters is shown in Figure \ref{fig-5.1a}, while Figure \ref{fig-5.1b} shows the variation in $\vert \hat{h}(\xi)\vert^{2} + \vert \hat{g}(\xi)\vert^{2}$.

\begin{table}[ht]
\caption{Filter coefficients for the filter bank of Example \ref{eg5.1}.} \label{table-5.1}

\centering
\begin{tabular}{c|c|c}
$n$ & $h$ & $g$ \\
\hline
    -2 &     -0.05125162  &    -0.05125162 \\
    -1 & \mm  0.25000000  &    -0.25000000 \\
 \mm 0 & \mm  0.60250325  & \mm 0.60250325 \\
 \mm 1 & \mm  0.25000000  &    -0.25000000 \\
 \mm 2 &     -0.05125162  &    -0.05125162 \\
\end{tabular}
\end{table}

\begin{figure}[ht]
    \centering
    
    \includegraphics[width=5.0in]{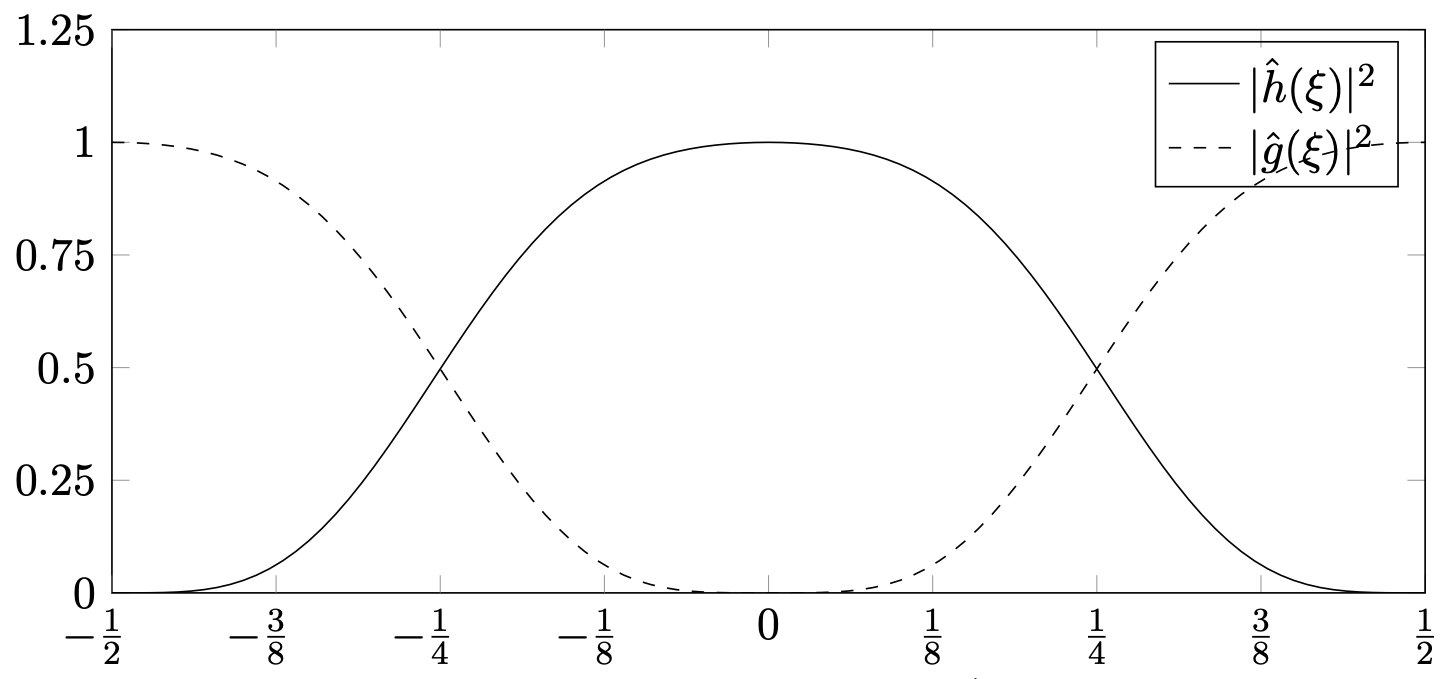}

    \caption{Frequency response for the filters $h$ and $g$ of Example \ref{eg5.1}.} \label{fig-5.1a}
\end{figure}

\begin{figure}[ht]
    \centering
    
    \includegraphics[width=5.0in]{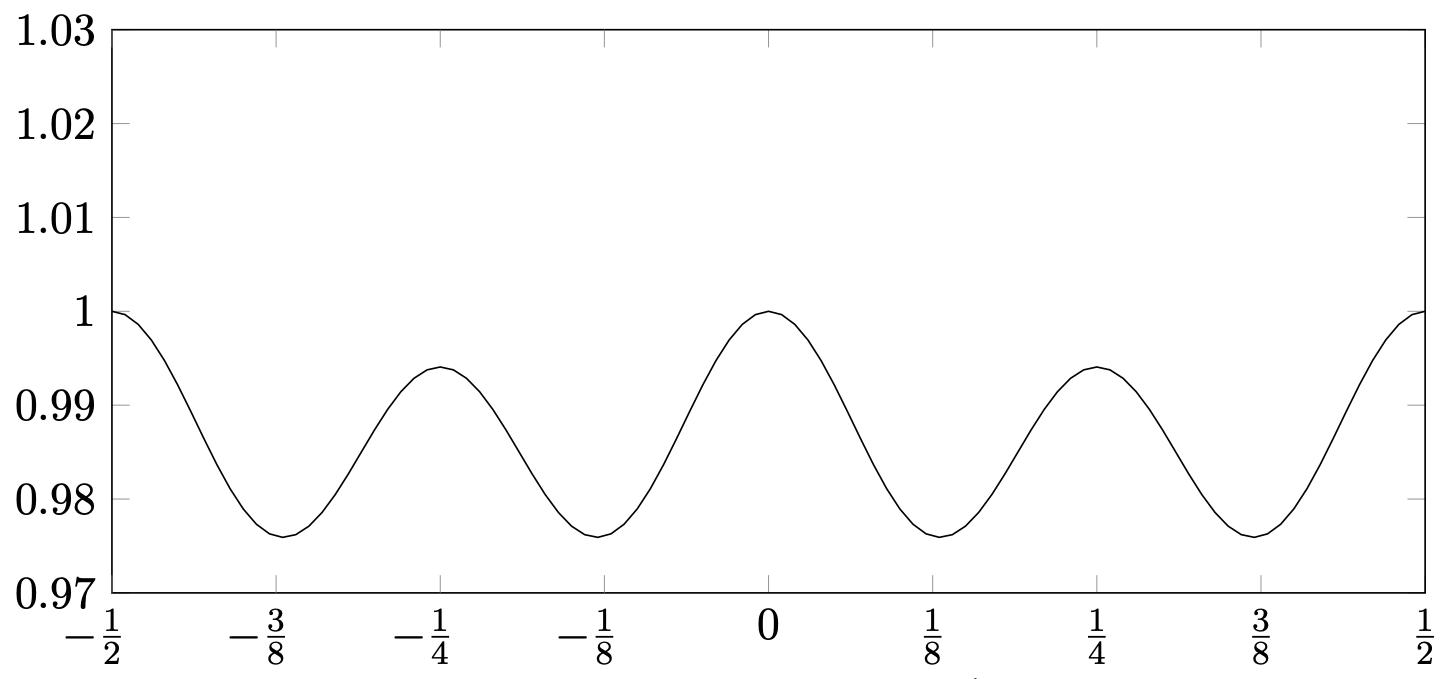}

    \caption{$\vert \hat{h}(\xi)\vert^{2}+\vert \hat{g}(\xi)\vert^{2}$ for the filters $h$ and $g$ of Example \ref{eg5.1}.} \label{fig-5.1b}
\end{figure}
\end{example}

The next example seeks to improve on Example \ref{eg5.1} by extending the support of the polynomial $p(\xi)$ in the filter design.  This will also add an additional degree of freedom, which can be used to tune the filters for smaller fluctuation in the quantity $\vert \hat{h}(\xi)\vert^{2}+\vert \hat{g}(\xi)\vert^{2}$.

\begin{example} \label{eg5.2}
This example will again consider a finitely supported low-pass filter $h$ of the form \eqref{haar-type} with $n=2$, but, here, the polynomial $p(\xi)$ will be assumed to take the form
$$ p(\xi) = (1+a+b) - a \cos{(2\pi \xi)} - b \cos{(4\pi \xi)},$$

\noindent
where $a$ and $b$ are positive constants.  Numerical experimentation led to the parameter values $a=0.32890122$ and $b=0.04248420$ and it is easy to verify that these values ensure $\vert p(\xi)\vert < 2$, which means the filters $h$ and $g$ will satisfy the hypotheses of Theorems \ref{Bessel-bound} and \ref{frame-bound} and thus the infinitely iterated shift-invariant filter bank associated with $h$ and $g$ is stable.  The filter coefficients for this filter bank are given in Table \ref{table-5.2}, while the frequency response of the filters is shown in Figure \ref{fig-5.2a}.  It is clear from Figure \ref{fig-5.2b} that $\vert \hat{h}(\xi)\vert^{2}+\vert \hat{g}(\xi)\vert^{2}$ shows very little oscillation with these filters and thus the stability constants (frame bounds) will hold up well under iteration.  In this case, the low- and high-pass filters again have identical time and frequency spreads with $\sigma_{n}^{2}\approx 0.305$ and $\sigma_{\omega}^{2}\approx 1.06$, leading to the time-frequency product $\sigma_{n}^{2} \sigma_{\omega}^{2} \approx 0.323$. 

\begin{table}[ht]
\caption{Filter coefficients for the filter bank of Example \ref{eg5.2}.} \label{table-5.2}

\centering
\begin{tabular}{c|c|c}
$n$ & $h$ & $g$ \\
\hline
    -3 &     -0.00531052  & \mm 0.00531052 \\
    -2 &     -0.05173370  &    -0.05173370 \\
    -1 & \mm  0.25531052  &    -0.25531052 \\
 \mm 0 & \mm  0.60346740  & \mm 0.60346740 \\
 \mm 1 & \mm  0.25531052  &    -0.25531052 \\
 \mm 2 &     -0.05173370  &    -0.05173370 \\
 \mm 3 &     -0.00531052  & \mm 0.00531052 \\
\end{tabular}
\end{table}

\begin{figure}[ht]
    \centering
    
    \includegraphics[width=5.0in]{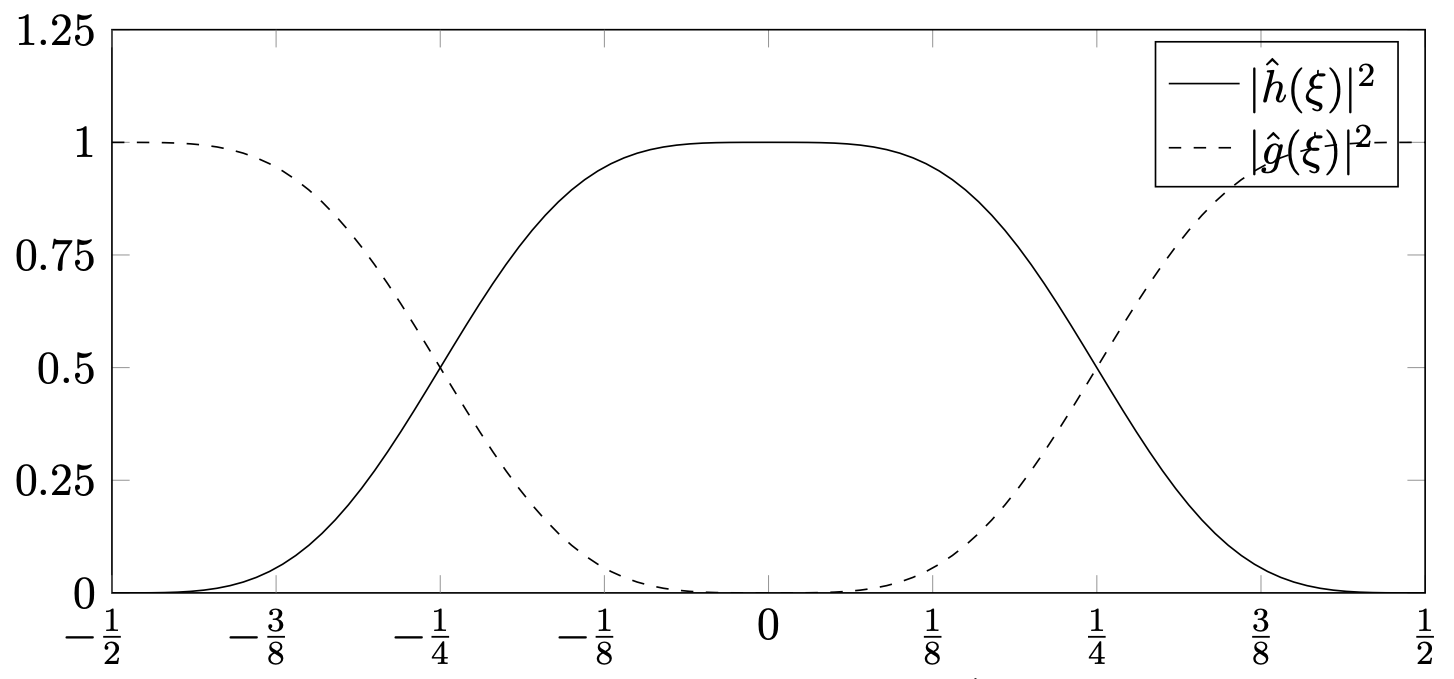}

    \caption{Frequency response for the filters $h$ and $g$ of Example \ref{eg5.2}.} \label{fig-5.2a}
\end{figure}

\begin{figure}[ht]
    \centering
    
    \includegraphics[width=5.0in]{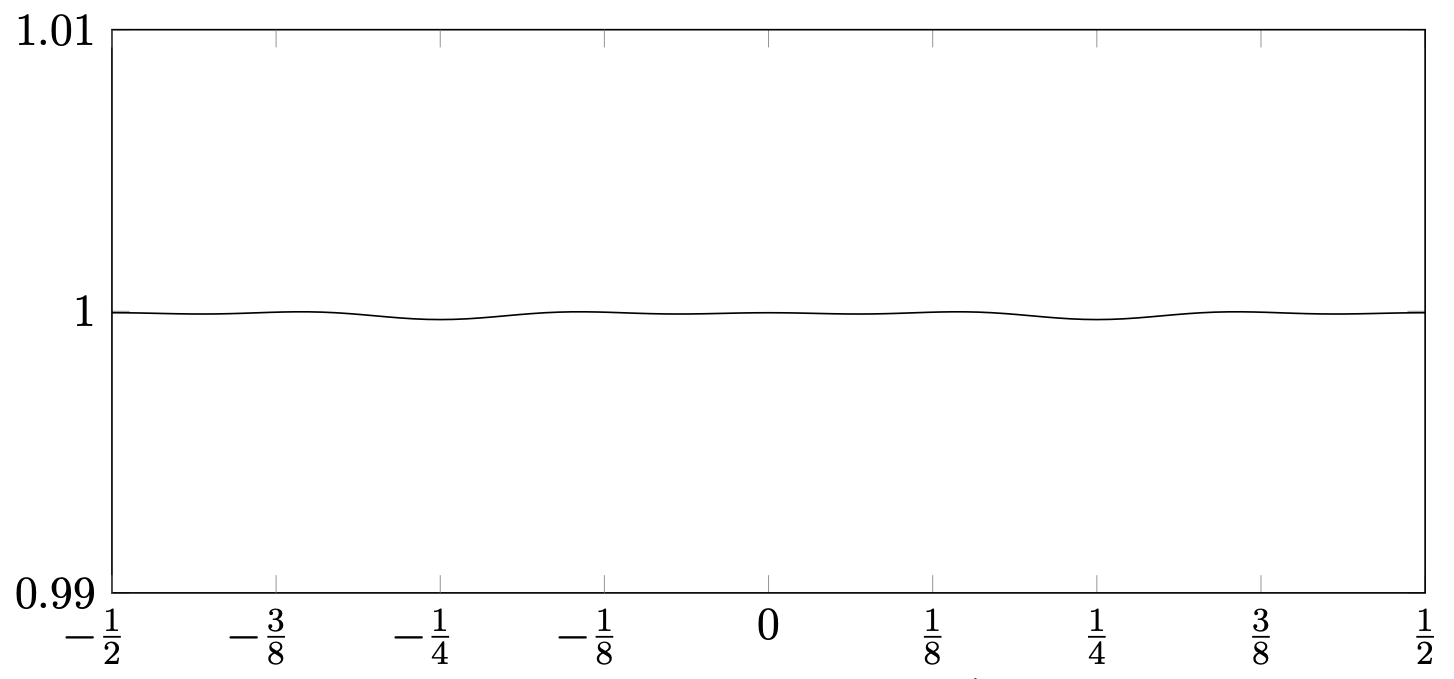}

    \caption{The quantity $\vert \hat{h}(\xi)\vert^{2}+\vert \hat{g}(\xi)\vert^{2}$ for the filters $h$ and $g$ of Example \ref{eg5.2}.} \label{fig-5.2b}
\end{figure}
\end{example}

\subsection{Filter Banks with Two High-Pass Filters} \label{twoHP}

This subsection is devoted to two filter bank designs in which two high-pass filters are employed in an attempt to achieve a finitely supported approximation of the following ideal filter arrangement,
\begin{equation} \label{ideal2HP}
\hat{h}(\xi) = \mathbbm{1}_{\lbrack -\frac{1}{4}, \frac{1}{4}\rbrack} (\xi), \quad  \hat{g}^{1}(\xi) = \mathbbm{1}_{\lbrack \frac{3}{8},-\frac{1}{4} \rbrack \cup \lbrack \frac{1}{4}, \frac{3}{8} \rbrack} (\xi), \quad \text{and} \quad \hat{g}^{2}(\xi) = \mathbbm{1}_{\lbrack -\frac{1}{2}, -\frac{3}{8}\rbrack \cup \lbrack \frac{3}{8}, \frac{1}{2}\rbrack} (\xi).
\end{equation}

\begin{example} \label{twoHP-1}
Consider the low-pass filter $h$ whose Fourier transform is given by
$$ \hat{h}(\xi) = 0.4+0.6 \cos{(2\pi \xi)} +0.1 \cos{(4\pi \xi)} - 0.1 \cos{(6\pi \xi)}.$$

\noindent
The high-pass filters will be of the form
$$ \hat{g}^{\ell}(\xi) = b_{\ell,0} + \sum_{k=1}^{4} b_{\ell,k} \cos{(2\pi k \xi)}, \quad \ell = 1,2,$$

\noindent
where the coefficients $b_{\ell,k}$ are chosen to interpolate specified points in the frequency response.  The coefficients of the high-pass filter $g^{1}$ are chosen to satisfy
$$ \hat{g}^{1}(0) = 0, \quad \hat{g}^{1}(5/32)=1/\sqrt{2}, \quad \hat{g}^{1}(\xi)(9/32)=1, \quad \hat{g}^{1}(3/8)=1/\sqrt{2}, \quad \text{and} \quad \hat{g}^{1}(1/2)=0,$$

\noindent
while the coefficients for $g^{2}$ are chosen according to
$$ \hat{g}^{2}(0) = 0, \quad \hat{g}^{2}(1/8)=0, \quad \hat{g}^{2}(\xi)(1/4)=0, \quad \hat{g}^{2}(3/8)=1/\sqrt{2}, \quad \text{and} \quad \hat{g}^{2}(1/2)=1.$$

\noindent
Numerical estimates of the filter coefficients are given in Table \ref{table-two-1} and the frequency response of each filter is shown in Figure \ref{fig-two-1}.  Meanwhile, the quantity $\vert \hat{h}(\xi)\vert^{2} + \vert \hat{g}^{1}(\xi)\vert^{2} + \vert \hat{g}^{2}(\xi)\vert^{2}$ for this filter bank is depicted in Figure \ref{fig-5.3b}.  The low-pass filter $h$ has $\sigma_{n}^{2}\approx 0.700$ and $\sigma_{\omega}^{2} \approx 0.543$ so that $\sigma_{n}^{2} \sigma_{\omega}^{2} \approx 0.380$.  The band-pass filter $g^{1}$ has $\sigma_{n}^{2}\approx 1.218$ and $\sigma_{\omega}^{2} \approx 0.244$ leading to $\sigma_{n}^{2} \sigma_{\omega}^{2} \approx 0.297$, while the high-pass filter $g^{2}$ has $\sigma_{n}^{2}\approx 1.091$ and $\sigma_{\omega}^{2} \approx 0.303$ with $\sigma_{n}^{2} \sigma_{\omega}^{2} \approx 0.331$.  The low-pass filter can be written in the form \eqref{haar-type} using $n=2$ with $p(\xi)$ given by
$$ p(\xi) = 0.2 e^{-2\pi i \xi} \left ( 1+6\cos{(2\pi \xi)}-2\cos{(4\pi \xi)} \right )$$

\noindent
and it is straightforward to verify that \eqref{p-condition} is satisfied with $s=1$.  It follows from Theorem \ref{Bessel-bound} that the infinitely iterated shift-invariant filter bank associated with $h, g^{1}, g^{2}$ is Bessel and it is routine to verify that these filters also satisfy the hypotheses of Theorem \ref{frame-bound}.  Hence, the infinitely iterated shift-invariant filter bank associated with $h, g^{1}, g^{2}$ is stable and by Theorem \ref{infinite2finite} the finitely iterated shift-invariant filter bank of order $j\in \mathbb{Z}$ associated with these filters will be stable with bounds independent of $j$.

\begin{table}[ht]
\caption{Coefficients of the high-pass filters in Example \ref{twoHP-1}.} \label{table-two-1}

\centering
\begin{tabular}{c|c|c|c}
$n$ & $h$ & $g^{1}$ & $g^{2}$ \\
\hline
    -4 &                &     -0.03511286 &     -0.02588834 \\
    -3 &    -0.05000000 & \mm  0.02810626 & \mm  0.00000000 \\
    -2 & \mm 0.05000000 &     -0.24357939 & \mm  0.12500000 \\
    -1 & \mm 0.30000000 &     -0.02810626 &     -0.25000000 \\
\mm  0 & \mm 0.40000000 & \mm  0.55738452 & \mm  0.30177670 \\
\mm  1 & \mm 0.30000000 &     -0.02810626 &     -0.25000000 \\
\mm  2 & \mm 0.05000000 &     -0.24357939 & \mm  0.12500000 \\
\mm  3 &    -0.05000000 & \mm  0.02810626 & \mm  0.00000000 \\
\mm  4 &                &     -0.03511286 &     -0.02588834 \\
\end{tabular}
\end{table}

\begin{figure}[ht]
    \centering
    \includegraphics[width=5.0in]{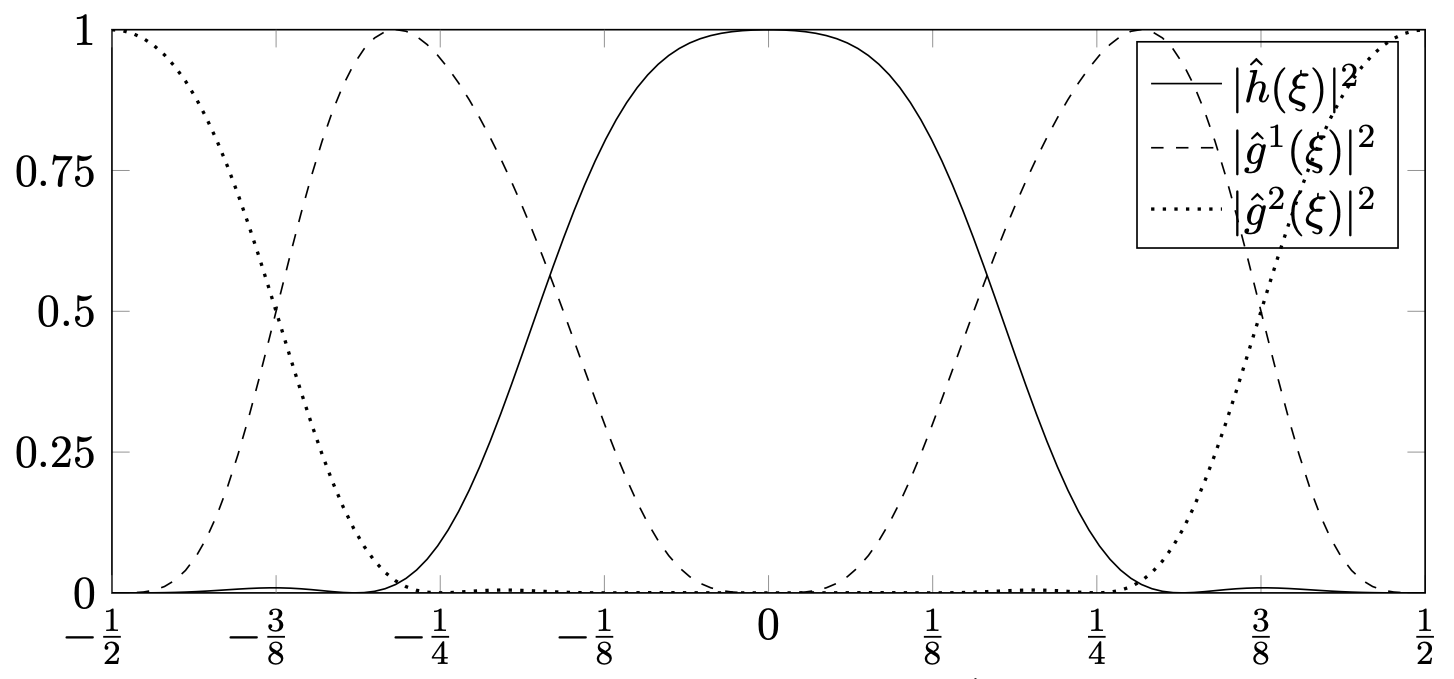}
    
    \caption{Frequency response for the filters $h, g^{1}, g^{2}$ of Example \ref{twoHP-1}.} \label{fig-two-1}
\end{figure}

\begin{figure}[ht]
    \centering
    
    \includegraphics[width=5.0in]{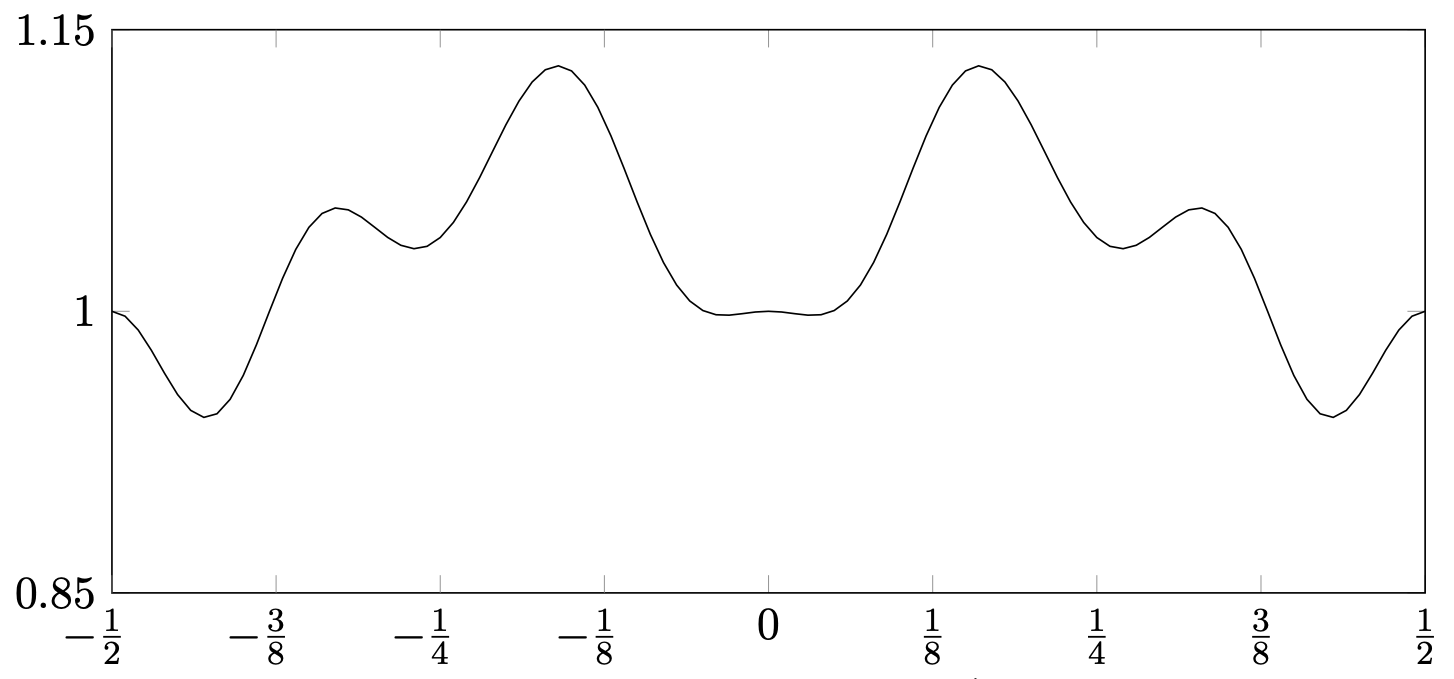}

    \caption{The quantity $\vert \hat{h}(\xi)\vert^{2}+\vert \hat{g}^{1}(\xi)\vert^{2}+\vert \hat{g}^{1}(\xi)\vert^{2}$ for the filters $h$, $g^{1}$, and $g^{2}$ of Example \ref{twoHP-1}.} \label{fig-5.3b}
\end{figure}
\end{example}

\begin{example} \label{twoHP-2}
It is also possible to construct finitely supported approximations of \eqref{ideal2HP} that employ perfect reconstruction filters satisfying \eqref{PReq}, which, by Proposition \ref{PerfectRecon}, guarantees that the associated finitely iterated shift-invariant filter banks will be Parseval for any number of iterations.  The idea is to employ Bezout's theorem \cite[Theorem 6.1.1]{Daubechies1992} in conjunction with the Riesz factorization lemma \cite[Lemma 6.1.3]{Daubechies1992} to construct finitely supported filers $h$, $g^{1}$, and $g^{2}$ such that $\vert \hat{h}(\xi)\vert^{2} + \vert \hat{g}^{1}(\xi)\vert^{2} + \vert \hat{g}^{2}(\xi)\vert^{2} = 1$.  The use of Bezout's theorem will be carried out in two steps.  Let $z=\sin^{2}{(\pi \xi)}$ and set $s_{0}=\frac{5\pi}{16}$ and $s_{1}=\frac{\pi}{4}$.  It follows from Bezout's theorem that there exist unique polynomials $p(z)$ and $q(z)$ of degree one and two, respectively, such that
$$ (s_{0}^{2}-z)^{2} (1-z) p(z) + z^{2} q(z) = 1.$$

\noindent
This factorization should be interpreted so that $\vert \hat{h}(z)\vert^{2} = (s^{2}-z)(1-z) p(z)$, while $\vert \hat{g}^{1}(z)\vert^{2}+\vert \hat{g}^{2}(z)\vert^{2} = z^{2} q(z)$.  This forces the Fourier transform of the low-pass filter to vanish at both $\xi=\pm \half$ and $\xi=\pm \frac{5}{16}$, while also introducing zeros for the high-pass filters at $\xi=0$.  The next stage of the factorization seeks unique polynomial solutions $r(z)$ and $s(z)$ of degrees two and one, respectively, such that
$$ (1-z)^{2} r(z) + z(s_{1}^{2}-z)^{2} s(z) = 1.$$

\noindent
This identity can be inserted in the term $z^{2} q(z)$ of the first factorization, producing
$$ (s_{0}^{2}-z)^{2} (1-z) p(z) + z^{2} (1-z)^{2} q(z) r(z) + z^{3} (s_{1}^{2}-z)^{2} q(z) s(z) = 1$$

\noindent
and leading to $\vert \hat{g}^{1}(z)\vert^{2} = z^{2} (1-z)^{2} q(z) r(z)$ and $\vert \hat{g}^{2}(z)\vert^{2} = z^{3} (s_{1}^{2}-z)^{2} q(z) s(z)$.  Thus the Fourier transform of $\hat{g}^{1}$ must vanish at $\xi=0$ and $\xi=\pm \half$, while the Fourier transform of $\hat{g}^{2}$ must vanish at $\xi=0$ and $\xi=\pm \frac{1}{4}$.  The Riesz factorization lemma is then used to express $\hat{h}(\xi)$, $\hat{g}^{1}(\xi)$, and $\hat{g}^{2}(\xi)$ as trigonometric polynomials.  Numerical estimates of the coefficients for $h$, $g^{1}$, and $g^{2}$ are provided in Table \ref{table-two-2} and the frequency response of each filter is illustrated in Figure \ref{fig-two-2}.  As one might expect, the time-frequency localization of this perfect reconstruction filter bank is inferior to that of the non-perfect reconstruction filter bank described in Example \ref{twoHP-1}.  The low-pass filter $h$ has $\sigma_{n}^{2}\approx 0.858$ and $\sigma_{\omega}^{2} \approx 0.674$ so that $\sigma_{n}^{2} \sigma_{\omega}^{2} \approx 0.578$.  The band-pass filter $g^{1}$ has $\sigma_{n}^{2}\approx 2.007$ and $\sigma_{\omega}^{2} \approx 0.1712$ leading to $\sigma_{n}^{2} \sigma_{\omega}^{2} \approx 0.344$, while the high-pass filter $g^{2}$ has $\sigma_{n}^{2}\approx 1.686$ and $\sigma_{\omega}^{2} \approx 0.669$ with $\sigma_{n}^{2} \sigma_{\omega}^{2} \approx 1.128$. 

\begin{table}[ht]
\caption{Filter coefficients for the filter bank of Example \ref{twoHP-2}.} \label{table-two-2}

\centering
\begin{tabular}{c|c|c|c}
$n$ & $h$ & $g^{1}$ & $g^{2}$ \\
\hline
    -4 &                 &     -0.03342562 & \mm  0.01271264 \\
    -3 &                 &     -0.10296278 & \mm  0.04169253 \\
    -2 &     -0.10956917 &     -0.05386255 & \mm  0.01150312 \\
    -1 & \mm  0.09694723 & \mm  0.33807931 &     -0.13643441 \\
\mm  0 & \mm  0.31919216 & \mm  0.13363824 &     -0.06718653 \\
\mm  1 & \mm  0.40305277 &     -0.36727027 & \mm  0.20830747 \\
\mm  2 & \mm  0.29037701 & \mm  0.02801366 &     -0.26150312 \\
\mm  3 &                 & \mm  0.13215374 & \mm  0.38643441 \\
\mm  4 &                 &     -0.07436373 &     -0.19552611 \\
\end{tabular}
\end{table}

\begin{figure}[ht]
    \centering
    \includegraphics[width=5.0in]{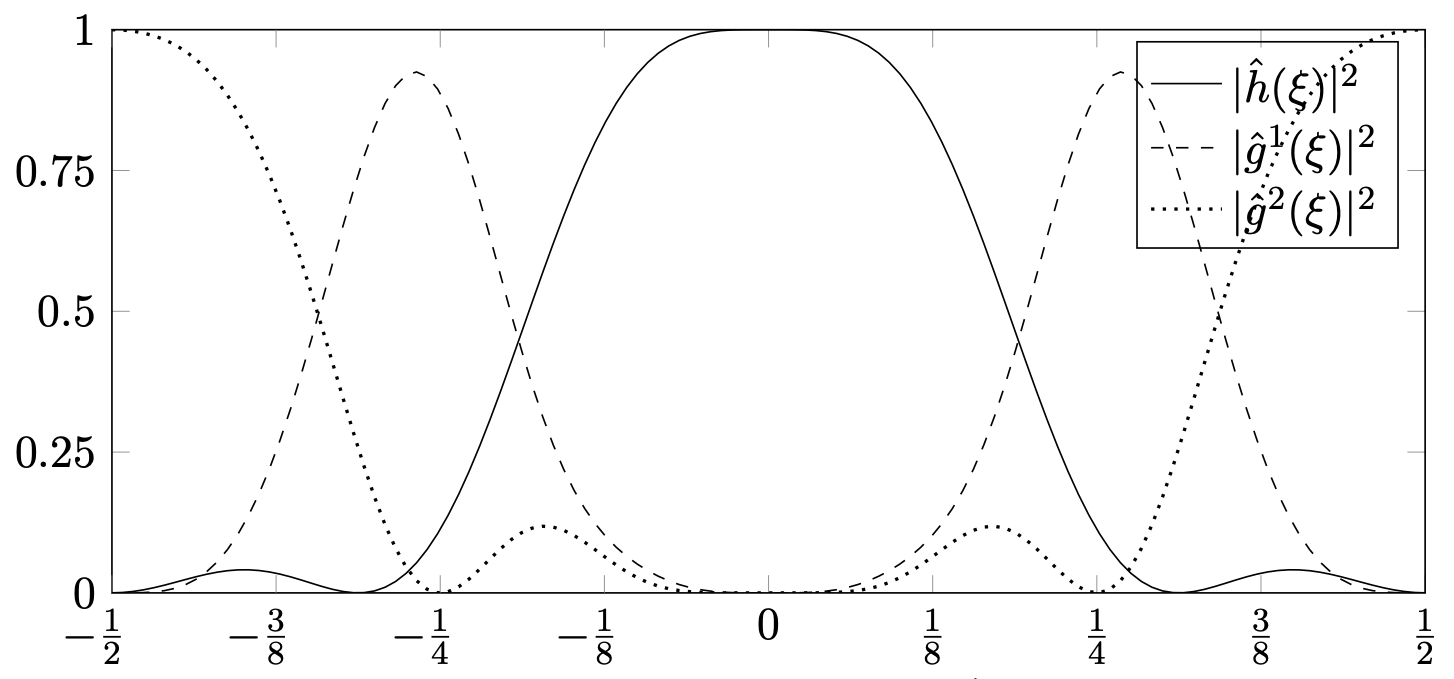}
    
    \caption{Frequency response for the filters $h, g^{1}, g^{2}$ of Example \ref{twoHP-2}.} \label{fig-two-2}
\end{figure}
\end{example}

\section{Separable Shift-Invariant Filter Banks in Two Dimensions} \label{separable}

Owing to the importance of discrete wavelet transforms in two dimensional applications, this section is devoted to a brief investigation of shift-invariant filter banks for $\ell^{2}(\mathbb{Z}^{2})$.  In particular, it will be shown that two-dimensional finitely iterated shift-invariant filter banks arising as the separable product of stable one-dimensional finitely iterated shift-invariant filter banks are also stable with bounds determined from the one-dimensional frame bounds.

The Fourier transform of $x\in \ell^{2}(\mathbb{Z}^{2})$ is defined by
$$ \hat{x}(\xi_{1},\xi_{2}) = \sum_{(k_{1},k_{2})\in \mathbb{Z}^{2}} x(k_{1},k_{2}) e^{-2\pi i (k_{1} \xi_{1} + k_{2} \xi_{2})}, \quad (\xi_{1}, \xi_{2}) \in \mathbb{T}^{2}.$$

\noindent
The convolution of $x,y \in \ell^{2}(\mathbb{Z}^{2})$ will be denoted by $x*y$ and is defined by
$$ (x*y)(k_{1},k_{2}) = \sum_{(n_{1}, n_{2})\in \mathbb{Z}^{2}} y(n_{1},n_{2}) x(k_{1}-n_{1},k_{2}-n_{2}), \quad (k_{1},k_{2})\in \mathbb{Z}^{2}.$$

\noindent
The upsampling operator $U:\ell^{2}(\mathbb{Z}^{2}) \rightarrow \ell^{2}(\mathbb{Z}^{2})$ acts on $x\in \ell^{2}(\mathbb{Z}^{2})$ by
$$ (Ux)(k_{1},k_{2}) = \begin{cases} x(n_{1},n_{2}), & k_{1}=2n_{1}, \; k_{2}=2n_{2} \\ 0, & \text{otherwise.} \end{cases}$$

\noindent
In accordance with the one-dimensional counterparts, these definitions lead to the identities 
$$\widehat{x*y}(\xi_{1}, \xi_{2}) = \hat{x}(\xi_{1}, \xi_{2}) \hat{y}(\xi_{1}, \xi_{2}), \quad (\xi_{1}, \xi_{2}) \in \mathbb{T}^{2}$$

\noindent
and 
$$\widehat{Ux}(\xi_{1}, \xi_{2}) = \hat{x}(2\xi_{1}, 2\xi_{2}), \quad (\xi_{1}, \xi_{2}) \in \mathbb{T}^{2}.$$

It will be convenient to alter the notation used for low- and high-pass filters in this discussion, as follows.  A typical two-dimensional filter bank will consist of filters $H^{0}, H^{1}, \ldots, H^{L}\in \ell^{2}(\mathbb{Z}^{2})$, $L\in \mathbb{N}$, with the convention that $H^{0}$ is a low-pass filter and $H^{\ell}$, $1\le \ell \le L$, are high-pass filters.  A low-pass filter $H^{0}$ will be assumed to satisfy $\widehat{H}^{0}(0,0)=1$, while high-pass filters are assumed to satisfy $\widehat{H}^{\ell}(0,0)=0$.  The notions of iterated low- and high-pass filters carry over to two dimensions without difficulty.  The iterated low-pass filter of order $j\in \mathbb{N}$ is defined by $H_{j}^{0} = H^{0}*UH^{0}*\cdots*U^{j-1}H^{0}$, while the iterated high-pass filters of order $j$ are defined by $H_{j}^{\ell} = H_{j-1}^{0}*U^{j-1}H^{\ell}$, $1\le \ell \le L$.  Following the notation established in Theorem \ref{infinite2finite}, given $x\in \ell^{2}(\mathbb{Z}^{2})$, define $x_{j} = x*H_{j}^{0}$ and $y_{j}^{\ell} = x*H_{j}^{\ell}$, $1\le \ell \le L$.

\begin{definition}
Let $H^{0}, H^{1}, \ldots, H^{L} \in \ell^{2}(\mathbb{Z}^{2})$ such that $H^{0}$ is a low-pass filter and $H^{\ell}$, $1\le \ell \le L$, are high-pass filters.  The finitely iterated shift-invariant shift-invariant filter bank of order $J\in \mathbb{N}$ associated with $H^{0}, H^{1}, \ldots, H^{L}$ is \emph{stable} provided that there exist positive constants $0<A\le B<\infty$ such that
$$ A \Vert x\Vert^{2} \le \Vert x_{J}\Vert^{2} + \sum_{\ell=1}^{L} \Vert y_{j}^{\ell}\Vert^{2} \le B \Vert x\Vert^{2}$$

\noindent
holds for each $x\in \ell^{2}(\mathbb{Z}^{2})$.
\end{definition}

The following result is a natural analog of Proposition \ref{stable-char} and will be stated without proof.

\begin{proposition} \label{stable2D}
Let $H^{0}, H^{1}, \ldots, H^{L} \in \ell^{2}(\mathbb{Z}^{2})$ where $H^{0}$ is a low-pass filter and $H^{1}, \ldots, H^{L}$ are high-pass filters.  The finitely iterated shift-invariant filter bank of order $J\in \mathbb{N}$ associated with $H^{0}, H^{1}, \ldots, H^{L}$ is stable if and only if there exist positive constants $0<A\le B<\infty$ such that
\begin{equation} \label{char2D}
A \le \vert \widehat{H}_{J}^{0}(\xi_{1}, \xi_{2}) \vert^{2} + \sum_{j=1}^{J} \sum_{\ell=1}^{L} \vert \widehat{H}_{j}^{\ell} (\xi_{1}, \xi_{2})\vert^{2} \le B, \quad \text{a.e.} \; \xi \in \mathbb{T}^{2}.
\end{equation}

\noindent
Moreover, the finitely iterated shift-invariant filter bank of order $J$ associated with $H^{0}, H^{1}, \ldots, H^{L}$ is Bessel with bound $B$ if and only if the right-hand inequality of \eqref{frame-char} holds.
\end{proposition}

Consider two one-dimensional filter banks consisting of filters $h^{0}, h^{1}, \ldots, h^{L}$ and $g^{0}, g^{1}, \ldots, g^{M}$, respectively, where $h^{0}$ and $g^{0}$ are one-dimensional low-pass filters and $h^{\ell}$, $1\le \ell \le L$, and $g^{m}$, $1\le m \le M$, are one-dimensional high-pass filters.  The \emph{separable product} of these two filter banks is a two-dimensional filter bank consisting of the filters $H^{\ell,m}\in \ell^{2}(\mathbb{Z}^{2})$, $0\le \ell \le L$ and $0\le m\le M$, defined according to
\begin{equation} \label{sep-filters}
H^{\ell,m}(k_{1},k_{2}) = h^{\ell}(k_{1}) \, g^{m}(k_{2}), \quad (k_{1}, k_{2}) \in \mathbb{Z}^{2}.
\end{equation}

\noindent
Notice that $\widehat{H}^{\ell,m}(\xi_{1},\xi_{2}) = \hat{h}^{\ell}(\xi_{1}) \hat{g}^{m}(\xi_{2})$.  The filter $H^{0,0}$ is the low-pass filter for the two-dimensional filter bank, satisfying $\widehat{H}^{0,0}(0,0)=1$, $\widehat{H}^{0,0}(\half,\xi_{2})=0$, and $\widehat{H}^{0,0}(\xi_{1},\half)=0$.  The remaining filters $H^{\ell,m}$, where either $\ell$ or $m$ is non-zero, are high-pass filters and each satisfies $\widehat{H}^{\ell,m}(0,0)=0$.

The main result of this section establishes a formula for the frame bounds of a two-dimensional shift-invariant filter bank which arises as the separable product of two stable one-dimensional shift-invariant filter banks.

\begin{theorem} \label{separableFB}
Let $h^{0}, h^{1}, \ldots, h^{L}\in \ell^{2}(\mathbb{Z})$ and $g^{0}, g^{1}, \ldots, g^{M}\in \ell^{2}(\mathbb{Z})$ where $h^{0}$ and $g^{0}$ are one-dimensional low-pass filters and and $h^{\ell}$, $1\le \ell \le L$, and $g^{m}$, $1\le m \le M$, are one-dimensional high-pass filters.  Assume that the finitely iterated shift-invariant filter banks of order $J\in \mathbb{N}$ associated with $h^{0}, h^{1}, \ldots, h^{L}$ and $g^{0}, g^{1}, \ldots, g^{M}$ are stable with bounds $0<A_{1}\le B_{1}<\infty$ and $0<A_{2}\le B_{2}<\infty$, respectively, independent of $J$.  If $H^{\ell,m}\in \ell^{2}(\mathbb{Z}^{2})$, $0\le \ell \le L$ and $0\le m\le M$, are defined according to \eqref{sep-filters}, then the finitely iterated shift-invariant filter bank of order $J\in \mathbb{N}$ associated with the filters $H^{\ell,m}$, $0\le \ell \le L$ and $0\le m\le M$, is stable with bounds $A_{1}A_{2}$ and $B_{1}B_{2}$.
\end{theorem}

\begin{proof}
In order to simplify notation for the indices of the various filters in the proof, let $I$ denote the set of indices given by
$$ I = \lbrace (\ell, m) : 0\le \ell \le L, \, 0\le m\le M\rbrace.$$

\noindent
The low-pass filter $H^{0,0}$ must often be separated from the high-pass filters in computations, so it will also be convenient to define $I_{0} = I\setminus \lbrace(0,0)\rbrace$.  In light of Proposition \ref{stable2D}, the goal of the proof is to establish lower and upper bounds on the quantity
$$ Q(\xi_{1},\xi_{2}) := \vert \widehat{H}_{J}^{0,0}(\xi_{1}, \xi_{2})\vert^{2} + \sum_{j=1}^{J} \sum_{(\ell,m)\in I_{0}} \vert \widehat{H}_{j}^{\ell,m}(\xi_{1}, \xi_{2})\vert^{2}.$$

\noindent
To aid in the clarity of the proof, it will also be helpful to write $Q(\xi_{1},\xi_{2})=\sum_{j=1}^{J} Q_{j}(\xi_{1},\xi_{2})$ where
$$ Q_{j}(\xi_{1},\xi_{2}) := \sum_{(\ell,m)\in I_{0}} \vert \widehat{H}_{j}^{\ell,m}(\xi_{1}, \xi_{2})\vert^{2}, \quad 1\le j\le J-1,$$

\noindent
and
$$ Q_{J}(\xi_{1}, \xi_{2}) : = \sum_{(\ell,m)\in I} \vert \widehat{H}_{J}^{\ell,m}(\xi_{1}, \xi_{2})\vert^{2}.$$

\noindent
The finitely iterated shift-invariant filter banks associated with $h^{0}, h^{1}, \ldots, h^{L}$ are stable with bounds $A_{1}$ and $B_{1}$ for any number of iterations, so, for $1\le j\le J$, it follows from Proposition \ref{stable-finite} that
\begin{equation} \label{hfilters}
A_{1} \le \vert \hat{h}_{j}^{0}(\xi) \vert^{2} + \sum_{k=1}^{j} \sum_{\ell=1}^{L} \vert \hat{h}_{k}^{\ell}(\xi)\vert^{2} \le B_{1}, \quad \text{a.e.} \; \xi \in \mathbb{T}.
\end{equation}

\noindent
Similar reasoning implies that
\begin{equation} \label{gfilters}
A_{2} \le \vert \hat{g}_{j}^{0}(\xi) \vert^{2} + \sum_{k=1}^{j} \sum_{m=1}^{M} \vert \hat{g}_{k}^{m}(\xi)\vert^{2} \le B_{2}, \quad \text{a.e.} \; \xi \in \mathbb{T},
\end{equation}

\noindent
for $1\le j\le J$.  Observe that
$$ Q_{J}(\xi_{1},\xi_{2}) = \sum_{(\ell,m)\in I} \vert \widehat{H}_{J}^{\ell,m}(\xi_{1}, \xi_{2})\vert^{2} = \left ( \sum_{\ell=0}^{L} \hat{h}_{J}^{\ell}(\xi_{1}) \right ) \, \left ( \sum_{m=0}^{M} \hat{g}_{J}^{m}(\xi_{2}) \right ), $$

\noindent
and applying \eqref{hfilters} and \eqref{gfilters} with $j=J$ leads to the estimate
$$ Q_{J}(\xi_{1},\xi_{2}) \le \left ( B_{1} - \sum_{j=1}^{J-1} \sum_{\ell=1}^{L} \vert \hat{h}_{j}^{\ell}(\xi_{1})\vert^{2} \right ) \, \left ( B_{2} - \sum_{k=1}^{J-1} \sum_{m=1}^{M} \vert \hat{g}_{k}^{m}(\xi_{2})\vert^{2} \right ). $$

\noindent
Similarly, notice that \eqref{hfilters} and \eqref{gfilters} can also be used to bound $Q_{j}$, leading to

\medskip
\resizebox{0.975\textwidth}{!}{
\begin{math}
\begin{aligned}
Q_{j}(\xi_{1},\xi_{2}) &= \sum_{m=1}^{M} \vert \widehat{H}_{j}^{0,m}(\xi_{1}, \xi_{2})\vert^{2} + \sum_{\ell=1}^{L} \vert \widehat{H}_{j}^{\ell,0}(\xi_{1}, \xi_{2})\vert^{2} + \sum_{(\ell,m)\in I_{0}} \vert \widehat{H}_{j}^{\ell,m}(\xi_{1}, \xi_{2})\vert^{2} \\
&\le \left ( B_{1} - \sum_{r=1}^{j} \sum_{\ell=1}^{L} \vert \hat{h}_{r}^{\ell}(\xi_{1})\vert^{2} \right ) \sum_{m=1}^{M} \vert \hat{g}_{j}^{m} (\xi_{2}) \vert^{2} + \left ( B_{2} - \sum_{s=1}^{j} \sum_{m=1}^{M} \vert \hat{g}_{s}^{m}(\xi_{2})\vert^{2} \right ) \sum_{\ell=1}^{L} \vert \hat{h}_{j}^{\ell} (\xi_{1}) \vert^{2} +\sum_{(\ell,m)\in I_{0}} \widehat{H}_{j}^{\ell,m}(\xi_{1}, \xi_{2})\vert^{2}
\end{aligned}
\end{math}}

\noindent
and summing this expression for $1\le j\le J-1$ leads to

\medskip
\resizebox{0.975\textwidth}{!}{
\begin{math}
\begin{aligned}
\sum_{j=1}^{J-1} Q_{j}(\xi_{1},\xi_{2}) &\le \sum_{j=1}^{J-1} \left \lbrack  \left ( B_{1} - \sum_{r=1}^{j} \sum_{\ell=1}^{L} \vert \hat{h}_{r}^{\ell}(\xi_{1})\vert^{2} \right ) \sum_{m=1}^{M} \vert \hat{g}_{j}^{m} (\xi_{2}) \vert^{2} + \left ( B_{2} - \sum_{s=1}^{j} \sum_{m=1}^{M} \vert \hat{g}_{s}^{m}(\xi_{2})\vert^{2} \right ) \sum_{\ell=1}^{L} \vert \hat{h}_{j}^{\ell} (\xi_{1}) \vert^{2} +\sum_{(\ell,m)\in I_{0}} \widehat{H}_{j}^{\ell,m}(\xi_{1}, \xi_{2})\vert^{2} \right \rbrack \\
&= B_{1} \sum_{j=1}^{J-1}  \sum_{m=1}^{M} \vert \hat{g}_{j}^{m} (\xi_{2}) \vert^{2} + B_{2} \sum_{j=1}^{J-1}  \sum_{\ell=1}^{L} \vert \hat{h}_{j}^{\ell} (\xi_{1}) \vert^{2} -\sum_{j=1}^{J-1} \sum_{r=1}^{j} \sum_{(\ell,m)\in I_{0}} \vert \hat{h}_{r}^{\ell}(\xi_{1})\vert^{2} \, \vert \hat{g}_{j}^{m} (\xi_{2}) \vert^{2} -\sum_{j=1}^{J-1} \sum_{s=1}^{j} \sum_{(\ell,m)\in I_{0}} \vert \hat{h}_{j}^{\ell}(\xi_{1})\vert^{2} \, \vert \hat{g}_{s}^{m} (\xi_{2}) \vert^{2} \\
& \qquad + \sum_{j=1}^{J-1} \sum_{(\ell,m)\in I_{0}} \vert \hat{h}_{j}^{\ell}(\xi_{1})\vert^{2} \, \vert \hat{g}_{j}^{m} (\xi_{2}) \vert^{2} \\
&= B_{1} \sum_{j=1}^{J-1}  \sum_{m=1}^{M} \vert \hat{g}_{j}^{m} (\xi_{2}) \vert^{2} + B_{2} \sum_{j=1}^{J-1}  \sum_{\ell=1}^{L} \vert \hat{h}_{j}^{\ell} (\xi_{1}) \vert^{2} - \sum_{j=1}^{J-1} \sum_{k=1}^{J-1} \sum_{(\ell,m)\in I_{0}} \vert \hat{h}_{j}^{\ell}(\xi_{1})\vert^{2} \, \vert \hat{g}_{k}^{m} (\xi_{2}) \vert^{2}.
\end{aligned}
\end{math}}

\noindent
Notice that the terms in $\sum_{j=1}^{J-1} Q_{j}(\xi_{1},\xi_{2})$ all cancel with terms in the above estimate of $Q_{J}(\xi_{1},\xi_{2})$ leading to the conclusion that
$$ Q(\xi_{1},\xi_{2}) = \sum_{j=1}^{J} Q_{j}(\xi_{1},\xi_{2}) \le B_{1} \, B_{2}, \quad \text{a.e.} \; \xi \in \mathbb{T}^{2},$$

\noindent
establishing the Bessel bound.  Analogous reasoning leads to the estimate $Q(\xi_{1},\xi_{2})\ge A_{1} \, A_{2}$ for almost every $\xi \in \mathbb{T}^{2}$, providing the lower bound.  The details of this argument will be left to the reader.
\end{proof}

\end{document}